\documentclass[11pt]{amsart}
\usepackage{amsmath,amssymb,amsbsy,amsfonts,amsthm,latexsym,amsopn,amstext,amsxtra,epic, euscript,amscd,indentfirst,verbatim,graphicx, hyperref,xcolor, setspace,tikz,tikz-cd}
\usepackage{tabularx} 
\usepackage{relsize} 
\usepackage[margin=1.2in]{geometry}
\usepackage{enumitem} 
\usepackage{hyperref} 
\usepackage{float} 

\usepackage{pgfplots}
\pgfplotsset{compat=newest}
\usepgfplotslibrary{fillbetween}

\newcommand{\MyScale}{1.5} 

\DeclareMathOperator{\cs}{cs}
\begin{document}

\newtheorem{theorem}{Theorem}
\newtheorem*{theorem*}{Theorem}
\newtheorem{conjecture}[theorem]{Conjecture}
\newtheorem*{conjecture*}{Conjecture}
\newtheorem{proposition}[theorem]{Proposition}
\newtheorem*{proposition*}{Proposition}
\newtheorem{question}{Question}
\newtheorem{lemma}[theorem]{Lemma}
\newtheorem*{lemma*}{Lemma}
\newtheorem{cor}[theorem]{Corollary}
\newtheorem*{obs*}{Observation}
\newtheorem{obs}{Observation}
\newtheorem{model}{Model}
\newtheorem{condition}{Condition}
\newtheorem{definition}{Definition}
\newtheorem*{definition*}{Definition}
\newtheorem{proc}[theorem]{Procedure}
\newcommand{\comments}[1]{} 
\def\Z{\mathbb Z}
\def\Za{\mathbb Z^\ast}
\def\Fq{{\mathbb F}_q}
\def\R{\mathbb R}
\def\N{\mathbb N}
\def\C{\mathbb C}
\def\k{\kappa}
\def\grad{\nabla}
\def\M{\mathcal{M}}
\def\S{\mathcal{S}}
\def\pt{\partial}
\def\pd{\partial}

\newcommand{\todo}[1]{\textbf{\textcolor{red}{[To Do: #1]}}}
\newcommand{\note}[1]{\textbf{\textcolor{blue}{#1}} \\ \\}

\title[Negative curvature constricts fundamental gaps]{Negative curvature constricts the fundamental gap of convex domains}
 \author[Iowa State University]{Gabriel Khan} 
 \author[]{Xuan Hien Nguyen}

\email{gkhan@iastate.edu}
\email{xhnguyen@iastate.edu}

\date{\today}

\maketitle 


\begin{abstract}
We consider the Laplace-Beltrami operator with Dirichlet boundary conditions on convex domains in a Riemannian manifold $(M^n,g)$, and prove that the product of the fundamental gap with the square of the diameter can be arbitrarily small whenever $M^n$ has even a single tangent plane of negative sectional curvature. In particular, the fundamental gap conjecture strongly fails for small deformations of Euclidean space which introduce any negative curvature. We also show that when the curvature is negatively pinched, it is possible to construct such domains of any diameter up to the diameter of the manifold.
The proof is adapted from the argument of Bourni et. al. \cite{bourni2022vanishing}, which established the analogous result for convex domains in hyperbolic space, but requires several new ingredients.
\end{abstract}

\section{Introduction}

We study the Laplace-Beltrami operator with Dirichlet boundary conditions on a geodesically convex domain $\Omega$ within a Riemannian manifold $M$. For any such domain with non-empty boundary, the operator has a discrete spectrum
  \[
  0 < \lambda_1(\Omega) < \lambda_2(\Omega) \leq \lambda_3(\Omega) \leq \ldots, 
  \]
with an accumulation point at infinity. Many geometric properties can be gleaned from the spectrum \cite{kac1966can} 
and a large body of work is dedicated to studying eigenvalues in Euclidean space and on Riemannian manifolds.

The \emph{fundamental gap} or \emph{spectral gap} is the difference $\lambda_2 (\Omega)- \lambda_1(\Omega)$. The quantity determines the rate at which positive solutions of the heat equation with Dirichlet conditions converge to the first eigenspace. In quantum mechanics, it  characterizes the difference in energy between the stable state, corresponding to the first eigenfunction, and the first excited state, corresponding to the second eigenfunction. Due to its relevance in physics and mathematics, this gap has been studied in depth, and one of the driving conjectures in this area was the \emph{fundamental gap conjecture} (see \cite{ashbaugh1989optimal,van1983condensation,yau1986nonlinear} and the survey article \cite{ashbaughfundamental}).

\begin{conjecture*}[Fundamental gap conjecture]
Let $\Omega \subset \mathbb R^n$ be a bounded convex domain of diameter $D$ and $V:\Omega \to \mathbb{R}$ a convex potential. Then the eigenvalues of the Schr\"odinger operator $-\Delta + V$ satisfy
  \begin{equation}
  \label{fundamental gap estimate}
  \lambda_{2}(\Omega)-\lambda_{1} (\Omega)\geq \frac{3 \pi^{2}}{D^{2}}. 
  \end{equation}
\end{conjecture*}

The quantity on the right-hand side of \eqref{fundamental gap estimate} is the spectral gap for an interval when $V=0$. In dimension 2, by taking narrower rectangles (and similarly in higher dimensions), it is possible to get arbitrarily close to the right-hand side of the inequality, so the lower bound is sharp. Various papers established estimates for the gap \cite{singer1985estimate,yu1986lower}, but were unable to obtain the sharp conjectural bounds. Finally, in 2011 this conjecture was proved by Andrews and Clutterbuck using a novel two-point maximum principle \cite{andrews2011proof}.

In the more general setting of Riemannian manifolds, there are a number of papers studying the fundamental gap on round spheres (see, e.g., \cite{lee1987estimate,wang2000estimation}). In recent work, Seto, Wang and Wei adapted the Andrews-Clutterbuck approach to show that geodesically convex domains in the round sphere $\mathbb{S}^n$ satisfy $\lambda_2-\lambda> \frac{3\pi^2}{D^2}$ \cite{seto2019sharp}. Fewer papers cover the fundamental gap on manifolds with non-constant curvature. In \cite{Odenetal}, Oden, Sung, and Wang prove a lower bound for the fundamental gap on a compact Riemannian manifold with nonempty boundary satisfying a rolling $R$-ball condition (see also, \cite{Olive-Rose-WW}) Furthermore, Tuerkoen, Wei and the present authors derive a fundamental gap estimate for surfaces whose curvature satisfies a strong positivity condition \cite{khan2022fundamental}.

Recently, the second named author and several collaborators showed that the fundamental gap conjecture does \emph{not} hold in hyperbolic space \cite{BCNSWW} \cite{bourni2022vanishing}.  Even more strikingly, they showed that it is not possible to bound the product of the gap and the square of the diameter at all. Even for horoconvex domains, this quantity goes to zero as the diameter gets large \cite{nguyen2021fundamental}. The main purpose of this paper is to extend this result to Riemannian manifolds where the sectional curvature is negative or has mixed signs.

We first prove the result for two dimensional manifolds whose curvature is negatively pinched. The proof in this case is simpler because it is possible to use comparison geometry, but it contains most of the essential ideas for the general case. The proof for $n$-dimensional manifolds with pinched negative curvature is discussed with the case where the sectional curvature has mixed signs. 


\begin{theorem}
\label{thm:negative curvature}
Let $(M^n, g)$ be a Riemannian manifold whose sectional curvature is negatively pinched and  suppose there exists a minimizing geodesic of length $D$. Then, for all $\epsilon$, there is a domain $\Omega$ of diameter $D$ which is geodesically convex and such that 
    \[
    \lambda_2(\Omega) - \lambda_1(\Omega) \leq \frac{\epsilon}{D^2}.
    \]
\end{theorem}

Because we do not assume that the metric is homogeneous, the techniques of \cite{bourni2022vanishing} do not apply immediately. Instead, we adapt the general strategy with estimates that do not rely on constant curvature.

Furthermore, we are able to prove a stronger result in which even a single tangent plane of negative curvature is enough to build domains whose fundamental gap is arbitrarily small.

\begin{theorem}
\label{thm:main-ndimensions}
Let $(M^n, g)$ be a smooth Riemannian manifold. Suppose that there is a point $p$ and a tangent plane $V$ at $p$ satisfying $\kappa(V)<0$, where $\kappa$ is the sectional curvature. Then, for all $\epsilon$, there is a domain $\Omega$ which is geodesically convex and such that 
    \[
    \lambda_2(\Omega) - \lambda_1(\Omega) \leq \frac{\epsilon}{D^2},
    \]
    where $D$ is the diameter of $\Omega$.
\end{theorem}

In Theorem \ref{thm:main-ndimensions}, $D$ cannot be arbitrary and is taken to be small relative to the $C^1$ norm of the curvature and the diameter of the manifold. 

 The main idea in the proofs is to use the negative curvature to create convex domains with small ``necks", which are small regions where the domain concentrates before expanding at either end (see Figure \ref{fig:Omega_r,L,t}). The neck acts as a narrow channel where the principal eigenfunction must be very small.  This phenomenon was used in the previous work on $\mathbb H^n$. 
 In the homogeneous case, a separation of variables reduces the problem of computing the fundamental gap to studying an ODE. Here, because the curvature is not constant, this separation of variables is no longer possible, so we must use several different ideas. First, we apply comparison geometry and integral estimates rather than pointwise estimates to capture the behavior of the first eigenfunction $h_1$. Secondly, the proof of \cite{bourni2022vanishing} benefits from the symmetry of the region and of $h_1$. Without it, we have to build a one-parameter family of domains to find one where our ansatz for the second eigenfunction is orthogonal to $h_1$. 

\section{An overview of the paper}

Because the proof involves a string of estimates whose purpose is not immediately clear, we begin by providing a high-level overview of the argument to give some intuition for each step.

\subsection{Negatively pinched curvature}
\label{ssec:negatively-pinched}

All domains we will construct are narrow tubes around a given geodesic, so we start the paper by setting a Fermi coordinate system and derive estimates on the metric in these coordinates in Section \ref{sec:fermi}. 

Sections \ref{Start of 2D proof}, \ref{Eigenfunction analysis 2D}, and \ref{Cutoff analysis 2D} are dedicated to the proof of Theorem \ref{thm:negative curvature} in two dimensions. In Section \ref{Start of 2D proof}, we define a one-parameter family of domains, which look like very short and wide rectangles. Due to the negative curvature, each of these domains flares out at least on one side. We use the widening to obtain a bound on the first eigenvalue, which will be small compared to the associated eigenvalue for a similar rectangle in Euclidean space.

In Section \ref{Eigenfunction analysis 2D}, we turn our efforts to bounding the principal eigenfunction $h_1$ through the narrowest part of the domain, which we call the ``neck". The crux of the proof is the analysis of the vertical integral 
 \[
 V(x_0) = \int_{\Omega\cap \{x=x_0\}} g^{xx} h_1^2 dy
 \]
seen as a function of $x_0$ and where $g^{xx}$ is a component of the inverse of the metric. Integration over slices reduces the problem to one dimension. Here, we derive exponential growth for $V$ rather than for $h_1$ through the neck. 

In order to understand the gist of the argument in Section \ref{ssec:Vgrowth}, it is helpful (even if wrong) to pretend that the metric is $\delta_{ij}$ at first. We use the notation $\sim$ for denoting things that behave mostly the same way. For example, $ V(x_0) \sim \int_{\Omega\cap \{x=x_0\}} h_1^2 dy$. Taking two derivatives of $V$, one would get
\[
\frac{\pd^2}{\pd x^2} V(x) \sim \int _{\Omega\cap \{x=x_0\}} \pd^2_{xx} (h_{1}^2) dy. 
\]
The boundary terms vanish because of the boundary condition on $h_{1}$. Expanding $\pd^2_{xx}(h_{1}^2) = 2 (\partial_x h_{1})^2 + 2 h_{1} \pd^2_{xx}h_{1}$ and replacing $\pd^2_{xx} h_{1}$ by $\Delta h_{1} - \pd^2_{yy} h_{1}$, we obtain, after using integration by parts for the second line,
  \begin{flalign*}
\frac{1}{2}\frac{\pd^2}{\pd x^2} V(x_0) &\sim \int _{\Omega\cap \{x=x_0\}} h_{1} \Delta h_{1} - h_{1} \pd^2_{yy} h_{1} + (\pd_x h_{1})^2 dy\\
  & \sim \int _{\Omega\cap \{x=x_0\}} - \lambda_1 h_{1}^2  + (\pd_{y} h_{1})^2 + (\pd_x h_{1})^2 dy\\
  & \geq \int _{\Omega\cap \{x=x_0\}}   (\pd_{y} h_{1})^2 - \lambda_1 h_{1}^2  dy\\
  & \geq C V(x_0).
\end{flalign*}
The integral $\int (\pd_y h_1)^2 dy$ is greater than the eigenvalue of the cross-slice  times $V$. The fact that the constant $C$ is positive is a tug between the first eigenvalue of the cross-slice $\Omega \cap \{x=x_0\}$ and the eigenvalue $\lambda_1$ of the domain. Because the necks are small, the former is very large and because the height of the domain expands quickly, the latter is smaller. From the estimate on $\lambda_1$ in Section \ref{Start of 2D proof}, we have that 
$C$ is of order $\frac{1}{r^2}$ whenever $x_0$ is less than $\sim L/9$. Thus the quantity $V(x_0)$ doubles rapidly through the neck.

After normalizing $h_1$ in terms of its $L^\infty$ norm, we recall the gradient estimate from \cite{arnaudon2020gradient} to show that the supremum of $h_1$ also doubles rapidly through the neck so the eigenfunction must be very small in both supremum norm and $L^2$ norm in the neck.
   
    In Section \ref{Cutoff analysis 2D}, we use these observations to find an approximation for the second eigenfunction whose Rayleigh quotient is very close to that of the principal eigenfunction. In \cite{bourni2022vanishing}, the authors built an ansatz for the second eigenfunction by switching the sign of $h_1$ in the neck with the help of a function $\psi$.  Thanks to the symmetry, any function that was odd with respect to the first variable was orthogonal to $h_1$ and therefore provided an upper bound for the second eigenvalue through the Rayleigh quotient. For this article, we still construct the ansatz $\psi h_1$ the same way but the resulting functions are generally not orthogonal to $h_1$. It is worth noting that our particular choice of function $\psi$ is different from the one used in hyperbolic space, which allows us to avoid establishing a more refined gradient estimate.
    
    Finally, we make use of the one-parameter family of convex domains to find an value of the parameter where $\psi h_1$ and $h_1$ are orthogonal. For the sake of exposition, we will use a simpler argument which requires a minimizing geodesic of length $2D$. With a more technical argument, it is possible to instead only assume there is a minimizing geodesic of length $D+\epsilon$. This argument is discussed in Subsection \ref{A new continuity argument}.

\subsection{Mixed curvature and higher dimensions}

In Section \ref{sec:Three dimensional case}, we extend the proof to manifolds of higher dimensions whose sectional curvature is either have negatively pinched or has a mixed sign. There are two main difficulties compared to the case of surfaces with negative curvature. First, in the two dimensional case, the boundary of the domains were given by smooth geodesics. In higher dimensions, the boundaries of convex hulls need not be smooth, a consequence of the fact that most Riemannian metrics do not admit totally geodesic hyper-surfaces of codimension one. Second, in the mixed curvature case we do not assume that the sectional curvature is negative definite, so we will need to handle tangent planes with positive curvature. It also precludes the application of standard comparison theorems, so instead of proving a neck phenomenon, we derive a \emph{flattening} lemma instead. Furthermore, to account for the positive curvature, we will make the domain narrower in the directions of negative curvature compared to those of positive curvature\footnote{Here, the curvature of a direction $V$ is the sectional curvature of the tangent plane spanned by $V$ and $\dot \gamma$.},  and take $L$ to be small, which we did not have to do in the previous case.

\begin{figure}[H]
    \centering
\includegraphics[width=.9\linewidth]{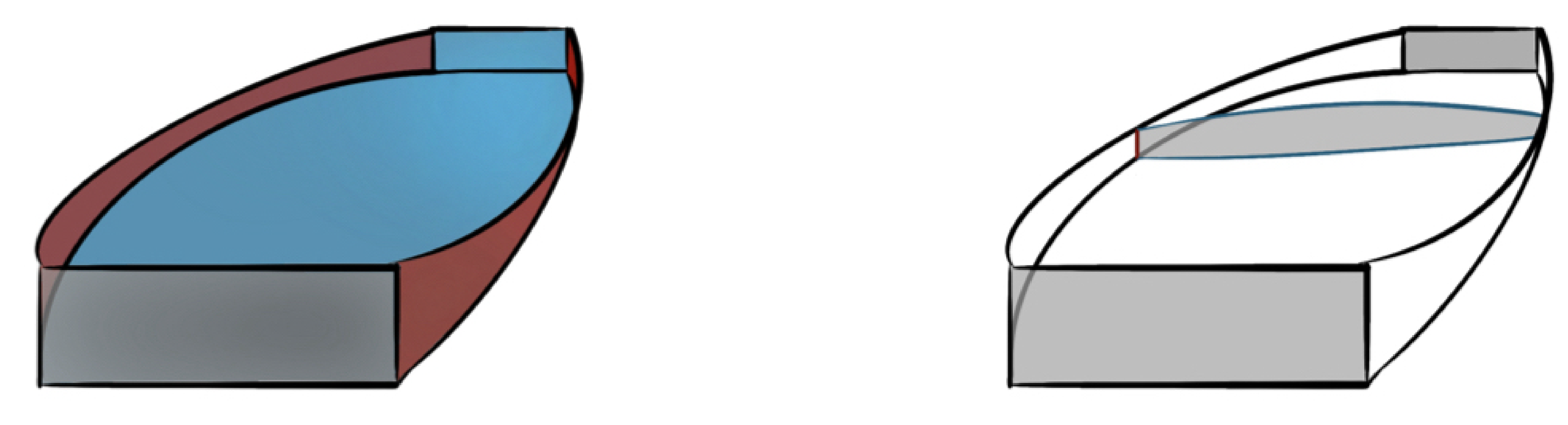}
\caption{A domain when the curvature has a mixed sign and the cross-slices used to define $V(x)$. \protect \footnotemark}
\label{fig:A model domain in 3D}
\end{figure}

\footnotetext{Throughout the paper, we will use the color red to indicate regions/directions of negative curvature and blue to indicate regions/directions of positive curvature.}

Once this is done, the proof is very similar to the case of negatively pinched curvature, which works by showing the principal eigenfunction must become very small in the flattened part. However, there is one final technical difficulty, which is that we must use a different one-parameter family of convex domains in the continuity argument.

\section{Fermi coordinates and estimates on the metric}
\label{sec:fermi}

We start by defining the coordinate system used and establishing bounds on the metric in these coordinates.

\subsection{The Fermi coordinates}
\label{ssec:coordinates}
Let $\gamma(x)$ be a geodesic segment of length $2L$ parametrized by arc length. Let $\{e_i\}_{i=1}^n$ be an orthonormal frame at $\gamma(0)$ with $e_1 = \gamma'(0)$.  We extend our frame $\{e_i\}_{i=1}^n$ to points $\gamma(x)$ by parellel transport.  In a tubular neighborhood of $\gamma(x)$, the chart $\phi: (-L, L) \times B_{r}^{n-1}$ for some small $r$ (not related to the $r$ of the rest of the article) given by  
  \[
  \phi(x,y)= \exp_{\gamma(x)} (  y^\alpha e_\alpha), \quad \alpha=2, \ldots, n
  \]
defines Fermi coordinates.

For $x$ fixed, $y \mapsto \phi(x,y)$ are normal coordinates on the $(n-1)$-dimensional manifold formed by geodesics rays perpendicular to $\gamma$ at $\gamma(x)$.

\subsection{The metric}
From the choice of coordinates, $g_{ij} (x,0) = \delta_{ij}$, where $g_{ij}$ is the metric in the Fermi coordinates chosen in Section \ref{ssec:coordinates}. In the two dimensional case, the indices go from 1 to 2, with the first coordinate being $x$ and the second one $y$.
\begin{proposition}
\label{prop:estimates-metric} Suppose that in a tubular neighborhood of $\gamma$, we have that the sectional curvature is bounded between two constants $K_1 \leq \kappa \leq K_2$.  Choosing an even smaller neighborhood of $\gamma$ then the one given in Lemma \ref{lem:metric-estimate} if necessary, we have
 \begin{equation}
    \label{eq:metric-estimates}
  |g_{ij} (x,y) - \delta_{ij}| \leq C \|y\|^2, 
 \end{equation}
where $\|y\|^2$ is the square of the distance from $(x,y)$ to $(x,0)$ and $C$ is a constant which depends on $K_1$ and $K_2$. Moreover, the Christoffel symbols and the first derivatives of the metric satisfy
 \begin{equation}
    \label{eq:gamma-estimates}
|\pd_i g_{kl} (x,y)|, |\Gamma_{ij}^k (x,y)| \leq C \|y\|,
 \end{equation}
 where $C$ is some constant which depends on $K_1$, $K_2$ and $|\nabla R|$. We also have for the second derivatives
 \begin{equation}
 \label{eq:pd2metric-estimates}
 |\pd^2_{xx}g_{kl}(x,y)| \leq C \|y\|^2
 \end{equation}
 where the constant $C$ now depends on $K_1$, $K_2$, $|\nabla R|$ and $|\nabla^2 R|$ in the chosen neighborhood. The size of this neighborhood depends on the third derivative of the curvature.
\end{proposition}
An immediate consequence of \eqref{eq:metric-estimates} are the following estimates on the inverse of the metric and its determinant: 
\begin{gather}
    \label{eq:g-dg} |g^{ij}(x,y) - \delta^{ij} | \leq C \|y \|^2, \quad |\pd_i g^{kl}(x,y)| \leq C \|y\|, \\
     \label{eq:detg}|\pd^2_{xx} g^{kl}(x,y)| \leq C \|y\|^2, \quad  \det(g) \leq 1 + C \|y \|^2.
\end{gather}

The Proposition \ref{prop:estimates-metric} is a corollary of Lemma \ref{lem:metric-estimate} in which we derive Taylor expansions of the metric in Fermi coordinates. The statement and proof of the Lemma \ref{lem:metric-estimate} and proof of the Proposition \ref{prop:estimates-metric} are given in Appendix \ref{ap:estimates-metric}.

 Throughout the rest of the paper, the constants will depend on the curvature tensor $R$, its derivatives $\nabla R$, $\nabla^2 R$, the dimension, and the constant $L$. We will suppress the notation $C(R,\nabla R, \nabla^2 R, n,L)$ and simply express such quantities as $C$. The constants are allowed to change from one line to the next, and will all be denoted by $C$.       
\section{Proof of Theorem \ref{thm:negative curvature}}

\label{Start of 2D proof}
In this section, we will start the proof of Theorem \ref{thm:negative curvature}. For the sake of exposition, we specialize to two dimensions (the higher dimensional case will be discussed in Subsection \ref{Higher dimensional case}). Our first task is to define the relevant domains and establish some important facts about their geometry.
 \subsection{The domains}
The constant $r$ is chosen small enough in order for the domains to be in the neighborhood of the geodesic $\gamma$ of length $2L$ given by Proposition \ref{prop:estimates-metric}. The constant $L$ is fixed, although in higher dimensional case, $L$ may need to be small in order to control the contribution of $\nabla R$. The one-parameter family of domains is just obtained by sliding cuts that are length $L$ apart (see Figure \ref{fig:Omega_r,L,t}). 

\begin{definition}[The domains $\Omega_{r,L}$ and $\Omega_{r,L,t}$]
\label{def:domains} Let $\sigma_x$ be the geodesic through $(x, 0)$ that is perpendicular to $\gamma$. Let $\gamma_y$ be the geodesic through $(0,y)$ that is perpendicular to $\sigma_0$. 
The domain $\Omega_{r,L}$ is the convex domain enclosed by $\gamma_{-r}$, $\gamma_r$, $\sigma_{-L}$ and $\sigma_L$. 
The domain $\Omega_{r, L, t}$, $t \in (0, L)$ is the convex domain enclosed by $\gamma_{-r}$, $\gamma_r$, $\sigma_{-L+t}$, and $\sigma_t$. The corners of $\Omega_{r, L, t}$ are denoted $P(t)$, $Q(t)$, $R(t)$, and $S(t)$ as in Figure \ref{fig:Omega_r,L,t}.   
\end{definition}

\begin{figure}[hbtp]
\begin{tikzpicture}[scale=\MyScale]
\node at (-100,0) {\includegraphics[width=.9\linewidth]{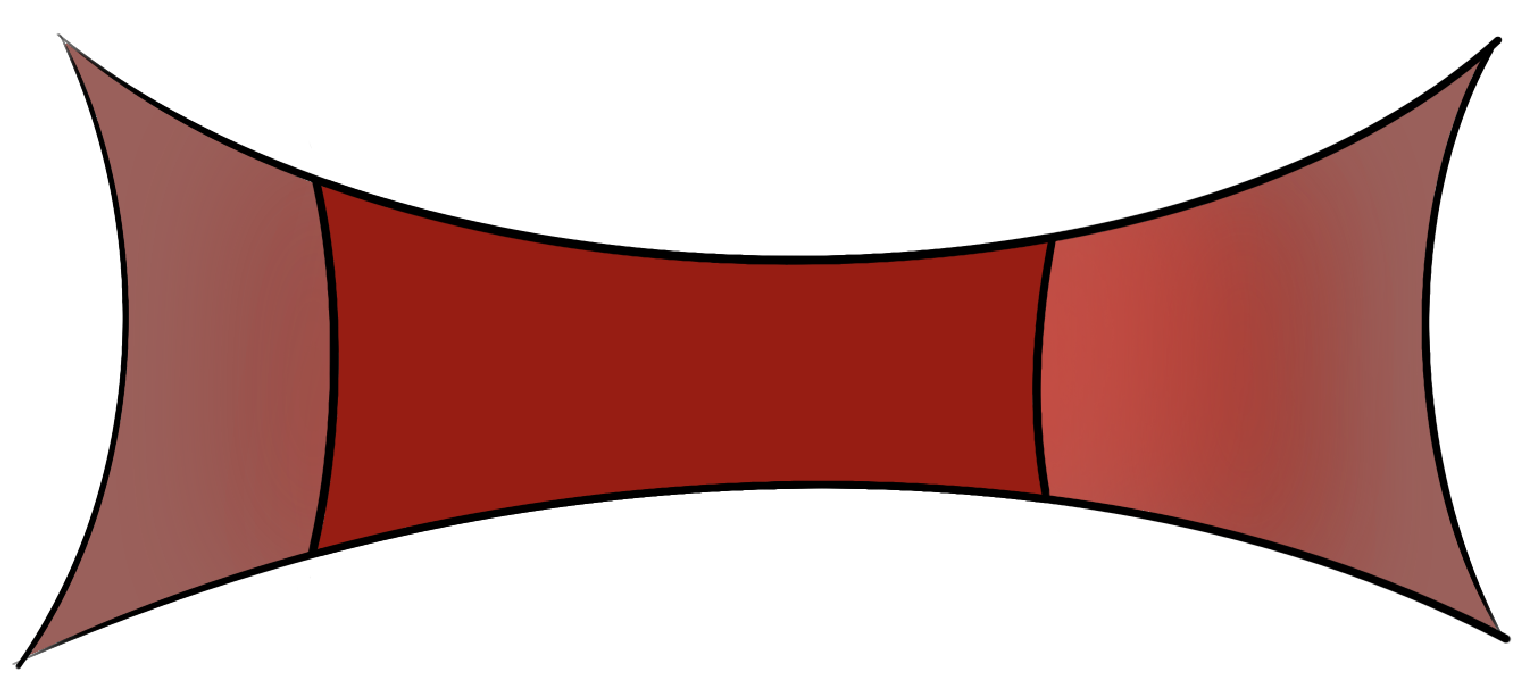}};
\node at (-98.3,-1.3) {$P(t)$};
\node at (-98.3,0.9) {$Q(t)$};
\node at (-102.4,1.2) {$R(t)$};
\node at (-102.4,-1.5) {$S(t)$};
\end{tikzpicture}
\caption{The domain $\Omega_{r, L, t}$}
\label{fig:Omega_r,L,t}
\end{figure}
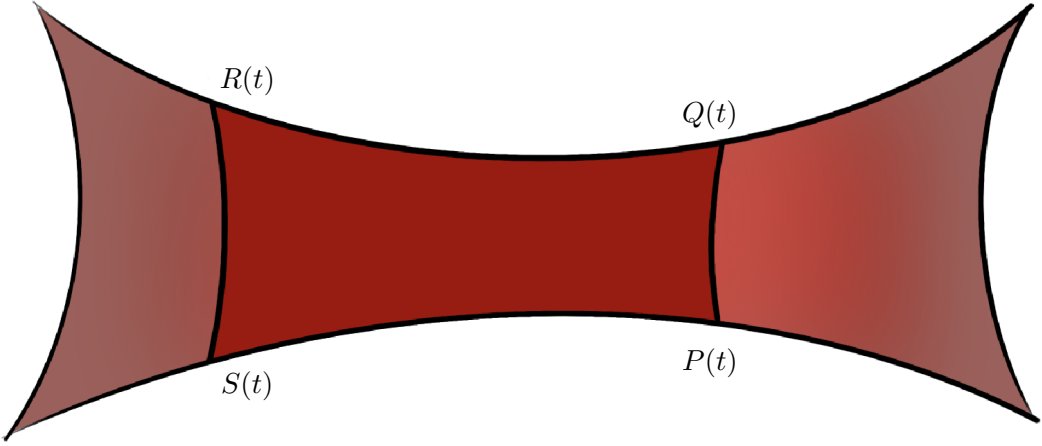
\subsection{The Rayleigh quotient}
We use the Rayleigh quotient to bound the difference of the first two eigenvalues. Recall that for a domain $\Omega$ and a function $h$ in $W^{1,2}_0$, the closure in $W^{1,2}$ of smooth functions compactly supported in $\Omega$, one defines the Rayleigh $R[h]$  quotient by  
\begin{equation*}
     \mathcal R[h] = \frac{\int_{\Omega}|\nabla h|^2}{\int_{\Omega} |h|^2}. 
\end{equation*}
The first and second eigenvalues of $\Omega$ are characterized by 
\begin{equation}
    \label{eq:eigenvalues-Rayleigh}
    \lambda_1 = \mathcal R[h_1]=\inf_{h \in W^{1,2}_0} \mathcal R[h], \qquad \lambda_2 = \inf_{h \perp h_1, h \in W^{1,2}_0} \mathcal R[h], 
\end{equation}
where $h_1$ is the first eigenfunction. Each infimum can also be taken over functions in $C^{1}_0(\Omega)$. 

\subsection{Comparing distances with domains in $\mathbb H^2$}
\label{sec:comparing_distances}

In this section, we prove that the domains $\Omega_{r, L, t}$   enclose a ``large-enough" set. By renormalizing the metric, we can assume from now on that the sectional curvature is bounded between $-K^2$ and $-1$ in this coordinate system.

Let $J$ be the Jacobi field along $\gamma$ with initial conditions $J(0)=r e_2 =r \frac{d}{ds} \sigma_0'(0)$ and $J'(0)=0$. Note that since the initial conditions of the Jacobi field are perpendicular to $\gamma'(0)$, $J(x)$ is perpendicular to $\gamma'(x)$. We use an extension of Rauch's Comparison Theorem (see \cite{doCarmo} p. 234 Theorem 4.9) and find that 
\[
r \cosh(x) \leq |J(x)| \leq \frac{r}{K} \cosh(K x)
\]
The Jacobi field represents the first degree variation of geodesic spread. Thus, for any $2\delta>0$, we can have
\begin{equation}
    \label{eq:vertical-length}
    (1-2\delta) r \cosh(t) \leq \text{dist}((t, 0), P(t)) \leq  (1+2\delta) \frac{r}{K} \cosh (Kt)
\end{equation}
provided $r$ is small enough depending on $\delta$ (and $L$), which is not a problem because $r$ will go to zero.

When $t \geq L/2$ the domain $\Omega_{r, L,t}$ contains all the Fermi coordinate points 
\begin{equation}
\label{eq:contained-rect}
    \mathcal W = \left\{
(x, y) \mid x \in [L/4, L/2], y \in \left[-(1-2\delta)r \cosh(L/4) , (1-2\delta)r\cosh(L/4) \right]\right\}. 
\end{equation}
The $\mathcal W$ is for ``wectangle" or rectangle in our Fermi coordinates. 
When $t \leq L/2$, we have the same bounds for $y$ but $x\in [-L/2, -L/4]$. From now on, let us assume that $t \leq L/2$. The case $t\geq L/2$ is treated similarly.  

\subsection{An upper bound on the principal eigenvalue}

We now derive an upper bound on the principal eigenvalue of $\Omega_{r,L,t}$. Because these domains contain the wectangle $\mathcal W$ as in \eqref{eq:contained-rect}, we can control the first eigenvalue of $\Omega_{r,L,t}$ by the first eigenvalue of $\mathcal W$ using monotonicity.  To estimate the latter, we appeal to \eqref{eq:eigenvalues-Rayleigh} and insert the function
\[
f(x, y) = \sin \left( \frac{\pi (x-L/4)}{L/4} \right) \cos \left( \frac{\pi y}{2(1-2\delta)r \cosh(L/4)} \right).
\]

 We denote by $\partial_x$ and $\partial_y$ the standard derivatives with respect to $x$ and $y$, respectively, and by $\nabla_x$ and $\nabla_y$ the covariant derivatives. Using \eqref{eq:g-dg} to estimate the inverse of the metric and Young's inequality to handle the term  $\partial_x f \partial_y f$, we obtain 
\begin{equation} \label{norm gradient squared bound}
|\nabla f|^2  \leq (1+C r^2) \left [ \left( {\partial_x} f \right)^2 +\left( {\partial_y} f \right)^2 \right ].
 \end{equation}
Note that the constant $C$ in \eqref{norm gradient squared bound} depends on $L$ because  $|y|$ is of order $r \cosh{L}$. We will integrate over $\mathcal W$ but also pass to coordinates $(x,y)$ and do the integration in $\R^2$. To distinguish between the two, $\mathcal W$ is a domain on our manifold and $W$ is a rectangle in the plane with the corresponding volume forms for each of the domains. Integrating inequality \eqref{norm gradient squared bound} over $\mathcal W$, and using \eqref{eq:detg} for the second line, we have that 
\begin{align*} 
\int_{\mathcal W} |\nabla f|^2 
& 
\leq \int_{W} (1+Cr^2) \left[\left( {\partial_x} f \right)^2 +\left( {\partial_y} f \right)^2 \right ] (\det g)\, dx\,  dy,\\
& 
\leq (1+Cr^2)\int_{W}  \left[\left( {\partial_x} f \right)^2 +\left( {\partial_y} f \right)^2 \right ] \, dx\,  dy,\\
&
= (1+C r ^2) \left(\frac{16 \pi^2}{L^2} + \frac{\pi^2}{4(1-2\delta)^2r^2 \cosh(L/2)} \right) \int_{W} f^2 \, dx \, dy,\\
&
\leq  \left(\frac{\pi^2}{4(1-2\delta)^2r^2 \cosh(L/2)}+ \frac{16 \pi^2}{L^2} + \frac{\pi^2}{\cosh(L/2)} + O(r^2)\right) \int_{\mathcal W} f^2,\\
& \leq  \left(\frac{\pi^2}{ 4(1-\delta)^2r^2\cosh(L/2)}\right) \int_{\mathcal W} f^2,
\end{align*}
for $r$ small enough. Therefore 
\begin{equation} \label{principaleigenvalueestimate}
    \lambda_1 (\Omega_t) \leq \lambda_1 (\mathcal W) \leq  \frac{ \pi^2}{4(1-\delta)^2 r^2\cosh(L/2)} 
\end{equation} 
for $r$ small. 

Before moving on, let us make two comments about this estimate.
\begin{enumerate}[label=(\roman{enumi})]
    \item For fixed $L$, as $r \to 0$, the expression on the right-hand side goes to infinity. Using the fact that $r$ is very small, with slightly more work it can be shown that the eigenvalue also goes to infinity. However, we will not need this fact so do not derive it here.
    \item Because of the factor $\cosh(L/2)$ in the denominator, so long as $\delta$ is sufficiently small the bound given by equation \eqref{principaleigenvalueestimate} is smaller than what one finds for a short rectangle in Euclidean space (where the fundamental gap conjecture holds). In other words, the eigenvalue can be controlled by the height in the ends, rather than the height in the neck.
\end{enumerate}

\section{Bounding the eigenfunction through the neck}

\label{Eigenfunction analysis 2D}

We now undertake the main step of the proof, which is to show that the principal eigenfunction is very small through the neck. The strategy is to prove an integral estimate then use a gradient estimate to translate it into pointwise bounds.

\subsection{Analyzing the vertical cross-slices: an $L^2$ bound.} \label{ssec:Vgrowth}

We consider Fermi coordinates as described in Section \ref{ssec:coordinates}. Let $\ell_-(x)$ and   $\ell_+(x)$ be the length of the geodesics connecting $(x, 0)$ to $P(x)$ and $(x,0)$ to $Q(x)$, respectively. We denote $\ell (x) = \max(\ell_+(x), \ell_-(x))$. Throughout the rest of the proof, we will only consider the first eigenfunction so will drop the subscript from $h_1$ and $\lambda_1$ unless there is some possibility for confusion. Therefore, let $h_{}$ be the first eigenfunction of $\Omega_{r,L,t}$ normalized so that $\|h_{}\|_{L^\infty} = 1$. 

Our goal now is to estimate the quantity
\begin{equation} \label{Vertical L^2 norms}
V(x) =    \int_{-\ell_-(x)}^{\ell_+(x)} g^{xx} h_{}^2 dy.
\end{equation}

In \cite{bourni2022vanishing}, the first eigenfunction was shown to increase exponentially through the neck for fixed $y$. This was done through a separation of variables and studying an explicit ordinary differential for the $x$ variable. Here, we derive the same result (exponential growth through the neck), but do so by considering the behavior of $V(x)$, which is heuristically the ``squared $L^2$-norm" of vertical slices. 

\begin{lemma}
\label{lem:second-derivative-V}
Let $V(x)$ be defined as in \eqref{Vertical L^2 norms} and let $\delta>0$ be an arbitrarily small constant. For $L > -2 \ln (1-\delta)$ and for $r$ sufficiently small depending on $\delta$, there is a positive constant $C(R, \nabla R, \nabla^2 R, L)$ so that for $|x| < L/9$,
\begin{equation}
\label{eq:pd2V-estimate}
\frac{\pd^2}{\pd x^2} V(x) \geq \frac{C}{r^2} V(x). 
\end{equation}
\end{lemma}
\begin{proof}
The main idea is outlined in Section  \ref{ssec:negatively-pinched}. The work to be done here is due to the fact that the metric is not Euclidean but well controlled. 

We first compute the second derivative of $V(x)$. As mentioned, the boundary terms vanish because $V$ is of second order in $h_{}$, which takes the value zero at $\ell_+(x)$ and $-\ell_-(x)$ by the Dirichlet boundary conditions. We have 
\begin{equation}
    \label{eq:VSecondDerivative}
    \frac{1}{2}\frac{\pd^2}{\pd x^2} V(x) = \int_{-\ell_-(x)}^{\ell_+(x)} \left( \frac{1}{2} \pd^2_{xx}(g^{xx}) h_{}^2 + 2 (\pd_xg^{xx}) h_{} \pd_x h_{} + g^{xx} (\pd_x h_{})^2 + g^{xx} h_{} \pd^2_{xx} h_{} \right) \, dy.
\end{equation}
The last term is replaced by an expression containing the Laplacian. Recall that $\Delta h = g^{ij}\pd^2_{ij} h - g^{ij}\Gamma_{ij}^k \pd_k h$ and that $h(\partial_{yy}^2 h)=\frac{1}{2} \partial_{yy}^2 (h^2) -(\partial_{y} h)^2$ for any function $h$. These facts and integration by parts give
\begin{align}
 \int_{-\ell_-(x)}^{\ell_+(x)} g^{xx} h_{}\partial^2_{xx} h_{} \, dy & = \int_{-\ell_-(x)}^{\ell_+(x)} h_{}\Delta h_{}   dy -  \int_{-\ell_-(x)}^{\ell_+(x)} g^{yy} h_{} \pd^2_{yy} h_{} \, dy  \nonumber \\
&  \ \ \ - 2\int_{-\ell_-(x)}^{\ell_+(x)} g^{xy} (\partial^2_{x y} h_{}) h_{} \, dy + \int_{-\ell_-(x)}^{\ell_+(x)} g^{ij}\Gamma_{ij}^k (\partial_k h_{}) h_{} \, dy \nonumber, \\
&= 
\label{eq:laplacian-gradienty}
 - \int_{-\ell_-(x)}^{\ell_+(x)} \lambda_1 h_{}^2   dy + \int_{-\ell_-(x)}^{\ell_+(x)}   g^{yy}(\partial_{y} h_{} )^2  +(\pd_y g^{yy}) h_{} \pd_y h_{}\, dy \\
&  \ \ \ - 2\int_{-\ell_-(x)}^{\ell_+(x)} g^{xy} (\partial^2_{x y} h_{}) h_{} \, dy + \int_{-\ell_-(x)}^{\ell_+(x)} g^{ij}\Gamma_{ij}^k (\partial_k  h_{})  h_{} \, dy \nonumber.
\end{align}
Recall that $\ell(x)$ or simply $\ell$ is $\max(\ell_-(x), \ell_+(x))$. 
The upper bound on the first eigenvalue will be used to estimate the first term of \eqref{eq:laplacian-gradienty}. For the second term, we recall that $|g^{ii} - 1| \leq C \ell^2$ thus  $(1-C\ell^2) \int_{-\ell_-(x)}^{\ell_+(x)} \phi^2 
\, dy \leq \int_{-\ell_-(x)}^{\ell_+(x)} g^{ii} \phi^2 \, dy\leq (1 + C \ell^2) \int_{-\ell_-(x)}^{\ell_+(x)} \phi^2 \, dy$ for any function $\phi$, where $i$ could be $x$ or $y$. We choose to leave out a term of order $C \ell$ to absorb some of the errors in \eqref{eq:VSecondDerivative} and \eqref{eq:laplacian-gradienty} so the second term is bounded as follows, where the constants $C_i$ have the same dependence as our constant $C$ and were just labelled for a bit of clarity, 
\begin{align}
\nonumber\int_{-\ell_-(x)}^{\ell_+(x)}   g^{yy}(\partial_{y} h_{} )^2 \, dy & \geq  (1 - C \ell^2) \int_{-\ell_-(x)}^{\ell_+(x)} (\partial_{y} h_{} )^2 \, dy = (1-C_1\ell + C_2 \ell) \int_{-\ell_-(x)}^{\ell_+(x)} (\partial_{y} h_{} )^2 \, dy \\ 
\nonumber&\geq \frac{(1-C_1 \ell) \pi^2}{4 \ell^2} \int_{-\ell_-(x)}^{\ell_+(x)} h_{}^2 \, dy + C_2 \ell \int_{-\ell_-(x)}^{\ell_+(x)} (\partial_{y} h_{} )^2 \, dy  \\
\label{eq:estimate-gradienty}
&\geq \frac{(1-C_3 \ell) \pi^2}{ 4 \ell^2} V(x)+ C_2 \ell \int_{-\ell_-(x)}^{\ell_+(x)} (\partial_{y} h_{} )^2 \, dy. 
\end{align}
Combining \eqref{eq:VSecondDerivative}, \eqref{eq:laplacian-gradienty}, \eqref{eq:estimate-gradienty}, and rearranging terms, we obtain
\begin{align}
\label{eq:first-line}
    \frac{1}{2}\frac{\pd^2}{\pd x^2} V(x) &\geq \left[\frac{(1-C_3\ell) \pi^2}{4 \ell^2} - \frac{\pi^2 }{4 (1-\delta)^2r^2 \cosh(L/2)} \right] V(x) \\
    \label{eq:second-line}
    &\ \ \ +\int_{-\ell_-(x)}^{\ell_+(x)}g^{xx} (\pd_x h_{})^2 + C_2 \ell (\pd_y h_{})^2\, dy \\
    \label{eq:third-line}
    &\ \ \  + \int_{-\ell_-(x)}^{\ell_+(x)} \frac{1}{2} \pd^2_{xx}(g^{xx}) h_{}^2 \, dy - 2\int_{-\ell_-(x)}^{\ell_+(x)} g^{xy} (\partial^2_{x y} h_{}) h_{} \, dy \\
    \label{eq:fourth-line}
&  \ \ \ + \int_{-\ell_-(x)}^{\ell_+(x)} g^{ij}\Gamma_{ij}^k (\partial_k  h_{})  h_{} + 2 (\pd_xg^{xx}) h_{} \pd_x h_{} + (\pd_yg^{yy}) h_{} \pd_y h_{}  \, dy.
\end{align}
The first line \eqref{eq:first-line} is what we seek. The coefficient in front of $V(x)$ is positive for $\delta$ small and $|x|< L/9$ by \eqref{eq:vertical-length}. The terms on the second line \eqref{eq:second-line} are positive and are used to absorb the error terms from the third and fourth lines \eqref{eq:third-line} and \eqref{eq:fourth-line} in what follows. 

\subsubsection{ Error terms} 

The first term in \eqref{eq:third-line} can be bounded using estimates of the metric \eqref{eq:pd2metric-estimates} and \eqref{eq:metric-estimates}:
 \[ 
   \left|\int_{-\ell_-(x)}^{\ell_+(x)} \frac{1}{2} \pd^2_{xx}(g^{xx}) h_{}^2 \, dy \right| \leq C \ell^2 V(x).
  \] 
The second term is handled with integration by parts  
\[
\left| \int_{-\ell_-(x)}^{\ell_+(x)} g^{xy} (\partial^2_{x y} h_{}) h_{} \, dy\right| = \left|  \int_{-\ell_-(x)}^{\ell_+(x)} -( \pd_y g^{xy})( \pd_x h_{} ) h_{} - g^{xy} \pd_x h_{} \pd_y h_{} \, dy \right| 
.
\]
The term involving $(\pd_x h_{}) h_{}$ is estimated like the ones in \eqref{eq:fourth-line} as shown below. For the last one, we use \eqref{eq:metric-estimates} and obtain $\left| \int_{-\ell_-(x)}^{\ell_+(x)} g^{xy} \pd_x h_{} \pd_y h_{} \, dy \leq C \ell^2 \int_{-\ell_-(x)}^{\ell_+(x)} (\pd_x h_{})^2 + (\pd_y h_{})^2 \, dy\right| \leq C \ell^2 \left(\int_{-\ell_-(x)}^{\ell_+(x)} g^{xx} (\pd_x h_{})^2 \, dy + (\pd_y h_{})^2 \, dy\right) $. This can be absorbed by the terms in \eqref{eq:second-line}. 
   The terms in \eqref{eq:fourth-line} are treated similarly using estimates from Proposition \ref{prop:estimates-metric} and the Peter-Paul's a.k.a. Young's inequality:
   \begin{align*}
    \left| \int_{-\ell_-(x)}^{\ell_+(x)} 
      (\pd_yg^{yy}) h_{} \pd_y h_{}  \, dy \right| & 
       \leq C \ell \int_{-\ell_-(x)}^{\ell_+(x)} |h_{} \pd_y h_{}| \, dy \\
      & \leq C \ell \int_{-\ell_-(x)}^{\ell_+(x)} \frac{C_2}{2C}(\pd_y h_{})^2 + \frac{C}{C_2} (h_{})^2 \, dy \\
      & \leq \frac{C_2}{2} \ell \int_{-\ell_-(x)}^{\ell_+(x)} (\pd_y h_{})^2 \, dy + C \ell V(x).
   \end{align*}
  Note that in this expression we have used $C_2$ to indicate that the constants do not cancel.
Rearranging terms and invoking \eqref{eq:vertical-length}, we find that for $\delta$ and $r$ sufficiently small and $|x| < L/9$, we have

    \begin{equation*}
        \partial^2_{xx} V(x) \geq  \frac{C}{r^2} V(x). \qedhere
    \end{equation*} 
\end{proof}

Integrating this differential inequality, we find that the $V(x)$  grows exponentially through the neck with doubling radius comparable to $r$ in the neck. 
Before moving on, let us note one consequence of this estimate which will be used at the very end of the proof.

\begin{lemma}\label{Needed for the continuity argument}
If $t$ is less than $\frac{L}{9}$, $V(x)$ is decreasing for $x >  -\frac{L}{9}$. Furthermore, there exists a $c>0$ so that $V(x) \ll \exp(-c/r)$ for $x >  -\frac{L}{9}$.
\end{lemma}
\begin{proof}
When $t$ is small (i.e., less than $\frac{L}{9}$), the domain is mostly to the left of the neck. From the fact that $V(x) \to 0$ as $x \to t$ because of the Dirichlet conditions, we must have that $V(x)$ is decreasing for $x$ sufficiently large (e.g., $x> -\frac{L}{9}$). The reason for this is that $V(x)$ is convex in the neck, and the neck of $\Omega_{L,r}$ includes the right edge of $\Omega_{L,r,t}$.

To obtain the bound on $V(x)$, we integrate Inequality \eqref{eq:pd2V-estimate}. 
\end{proof}
Similarly, if $t$ is large, we have that $V(x)$ is increasing for $x$ in the neck. This gives a quantitative estimate to show that if the  domain is mostly on one side of the neck, the bulk of the principal eigenfunction must also be on that side.

\subsection{Supremum bounds through the neck}
We show pointwise bounds for the first eigenfunction in the neck region by appealing to a gradient bound on Dirichlet eigenfunctions in bounded domains. 

\subsubsection{An $L^\infty$ bound on the gradient}

The work of Arnaunden et al. \cite{arnaudon2020gradient} shows that for any eigenfunction $h$ of the Laplace operator on $\Omega$ with Dirichlet boundary conditions and corresponding eigenvalue $\lambda$, we have the estimate
\begin{equation}\label{gradient bounds}
 \| \nabla h \|_\infty \leq \|h \|_\infty \inf_t c(t) e^{\lambda t},
\end{equation}
where the function $c(t)$ is defined as
\[c(t) = 9.5 \alpha_{0}+\frac{2 \sqrt{\alpha_{0}}\left(1+4^{2 / 3}\right)^{1 / 4}\left(1+5 \times 2^{-1 / 3}\right)}{(t \pi)^{1 / 4}}+\frac{\sqrt{1+2^{1 / 3}}\left(1+4^{2 / 3}\right)}{2 \sqrt{t \pi}}, \]
and $$
\alpha_{0}=\frac{1}{2} \max \{\theta, \sqrt{(n-1) K}\},
$$
where $-K$ is a lower bound for the Ricci curvature on $\Omega$ and $-\theta$ is a lower bound on the mean curvature of $\partial \Omega$. Because all our domains $\Omega$ are convex, the mean curvature is non-negative, so $\alpha_0 \leq \frac{1}{2}\sqrt{(n-1) K}$.  Taking $t = \frac{1}{\lambda}$ in \eqref{gradient bounds}, we have that
$ \| \nabla h \|_\infty \leq (C_1  + C_2\sqrt{\lambda}) \|h \|_\infty$. 
Throughout the rest of the paper, we will ignore the $C_1$ term in this estimate, since our bounds on the principal eigenvalue \eqref{principaleigenvalueestimate} will be very large. Thus
  \begin{equation}
      \label{eq:h1-gradient-estimate}
      \| \nabla h_1 \|_{\infty} \leq C \frac{\| h_1 \|_{\infty}}{r \cosh(L/4) }.
  \end{equation}

We can now show that the first eigenfunction $h_{}$ is small in the neck $N:= \{(x, y) \mid |x| < L/18 \}$. In order for our domains to contain $N$, we will consider $\Omega_{r, L, t}$ for $t \in [L/9, 8L/9]$.

\begin{lemma}
\label{lem:h1-small-in-neck}
Let  $t \in [L/9, 8L/9]$ and let $h_{}$ be the first eigenfunction of the Laplacian with Dirichlet boundary conditions on $\Omega_t=\Omega_{r,L,t}$ normalized so that $\| h_{} \|_{L^{\infty}(\Omega_{r,L,t})} =1$.  We have 
\begin{equation}
    \label{eq:h1-bound-neck}
\|h_{} \|_{L^{\infty}(N)} \leq C \exp( - c/r)
\end{equation}
for some constants $c$ and $C$ depending on $\delta, L, R, \nabla R, \nabla^2 R$. 
\end{lemma}

\begin{proof}

A crude upper bound for $V(x)$ with \eqref{eq:g-dg} and \eqref{eq:vertical-length} is 
  \begin{equation}
      \label{eq:V-upperbound}
  V(x) \leq  (1+C \ell(x)^2) 2 \ell(x) \leq 2 (1+ C \ell(x)^2) (1+2 \delta) r \exp (2L/9) \leq  C r 
  \end{equation}
  for $r$ small enough and $|x| \leq L/9$. For a lower bound on $V(x)$, we let $a=\sup_{y} h(x_0,y)$, consider the case where the eigenfunction decays from $a$ as quickly as possible given the gradient estimate, and use  \eqref{eq:h1-gradient-estimate} for the last inequality to obtain
\begin{equation}
\label{eq:V-lowerbound}
 V(x_0) \geq  2 (1 - C \ell^2(x_0))  \int_0^{a/\|\nabla h\|_{\infty}} \left( \|\nabla h \|_\infty y \right)^2 \,dy 
 = \frac{2(1 - C \ell^2(x_0))}{3} \frac{a^3}{\| \nabla h \|_{\infty}}  
 \geq   C a^3 r. 
\end{equation}

Given $x_0$ with $|x_0| < L/18$, we compare the function $V$ to $f(x) =  V(x_0) \cosh\left(\tfrac{\sqrt C}{r}(x-x_0) \right)$ for $C$ as in \eqref{eq:pd2V-estimate}, which is the solution to $f'' = \frac{C}{r^2} f$ with initial conditions $f(x_0) = V(x_0)$ and $f'(x_0) =0$. If $V'(x_0)\geq 0$, we pick $x_1=x_0 + L/18$ otherwise choose $x_1 = x_0 - L/18$. We have 
    \[
    V(x_1) \geq f(x_1) \geq \frac{V(x_0)}{2} \exp\left(\tfrac{\sqrt{C}L}{18r}\right).
    \]
Combining this with \eqref{eq:V-upperbound} for $V(x_1)$ and \eqref{eq:V-lowerbound} for $V(x_0)$, we get
    \[
    Cr \exp(-3c/r) \geq V(x_0) \geq C r \left(\sup_{y} h(x_0, y)\right)^3.
    \]
Since $x_0$ was arbitrary as long as $|x_0| < L/18$, we  get \eqref{eq:h1-bound-neck} as desired. 
\end{proof}

As a consequence of \eqref{eq:h1-bound-neck}, we find that the supremum of $h$ cannot happen when $x$ is too small, so the supremum of $h$ occurs outside of the neck.

Before concluding this section, it is worthwhile to take stock of what we have established.

\begin{enumerate}[label=(\roman{enumi})]
    \item The principal eigenfunction ``doubles" rapidly through the neck of the domain, which is to say that the size of $h$ doubles at a length scale of $r$ in the horizontal direction.
    \item The eigenfunction is extremely small, both in terms of the $L^2$-norm and supremum norm in this region.
    \item When $t$ is sufficiently small, the supremum of $h$ and the vast majority of its $L^2$ mass lies to the left of the neck. On the other hand, when $t$ is sufficiently large, the supremum of $h$ and vast majority of its $L^2$ mass the lies to the right of the neck.
\end{enumerate}

\section{Applying a cut-off function through the neck}

\label{Cutoff analysis 2D}

In order to show that the fundamental gap vanishes, we find another function whose Rayleigh quotient is very close to that of $h$. We do so in two steps: first consider a function $v:\mathbb{R}^+ \to \mathbb{R}^+$ which rapidly vanishes as the input goes to zero, then construt a function $\psi(x)$ that satisfies $ |\nabla \psi| \leq \frac{2}{v(r)}$ and 
  \[
  \psi(x) = \begin{cases}
        1 & x> v(r) \\
        -1 & x< -v(r)
        \end{cases}
   \]

In two dimensions, it is possible to complete the rest of the proof using $v(r) = r^4$. For the argument to work independent of the dimension, we take 
 \begin{equation} \label{Cut-off regime}
     v(r) = \exp(-\delta/r )
 \end{equation}
 for $\delta$ sufficiently small. Note that the function $\psi$ differs from the one in \cite{bourni2022vanishing} in that it switches signs over a much smaller region. This allows us to avoid deriving integral estimates for the gradient of $h$ through the neck.

    \begin{figure}[htbp]
        \centering
        \includegraphics[width=.9\linewidth]{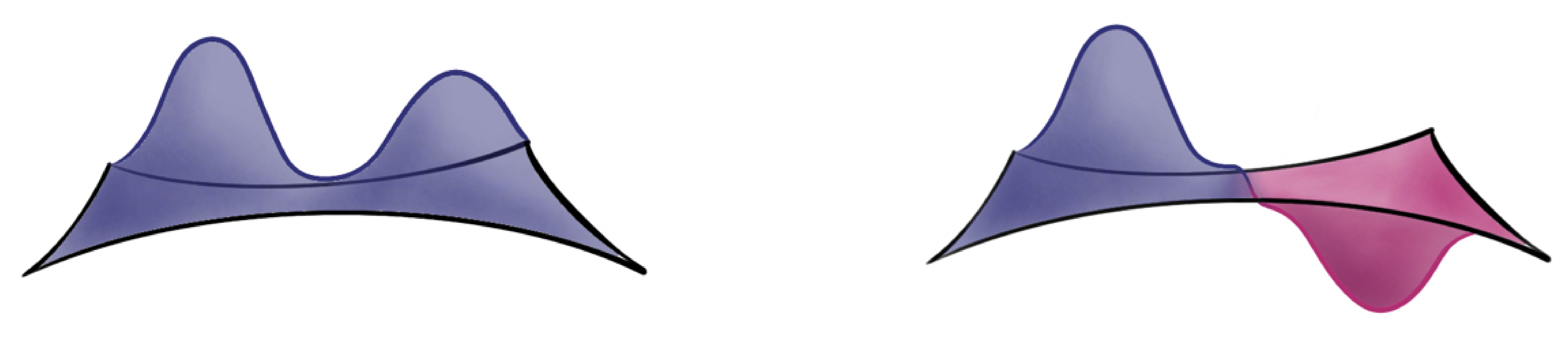}
        \caption{The graphs of the first eigenfunction $h$ and of the ansatz  $\psi h$.}
        \label{fig:handsh}
    \end{figure}
    
The difference in the Rayleigh quotients of $\psi  h$ and $h$ is given by 

\begin{equation}
\label{eq:diff-RQ}
\mathcal R[\psi h] - \mathcal R[h]= \left \vert \frac{\int_{\Omega_t} |\nabla  (\psi h)|^2}{\int_{\Omega_t} |\psi h|^2} - \frac{\int_{\Omega_t} |\nabla h|^2}{\int_{\Omega_t} |h|^2} \right \vert.
\end{equation}
   It is supported within the neck  $N_r = \{ (x,y) \in \Omega_t ~|~  -v(r) <x< v(r) \},$  so setting \[A = \int_{\Omega_t \backslash N_r} |\nabla h|^2 \textrm{ and } B = \int_{\Omega_t \backslash N_r} |h|^2,\]
   the expression in \eqref{eq:diff-RQ} simplifies to the following
\begin{eqnarray*}
\left \vert \frac{A \left ( \int_{N_r} h^2 -(\psi h)^2 \right)}{\left( B + \int_{N_r} (h)^2 \right) \left( B + \int_{N_r} (\psi h)^2 \right) } +\frac{B \left ( \int_{N_r} |\nabla (\psi h)|^2 - |\nabla h|^2 \right)}{\left( B + \int_{N_r} (h)^2 \right) \left( B + \int_{N_r} (\psi h)^2 \right) } \right \vert. 
\end{eqnarray*}
We consider the two terms separately, and bound each of them. Define 

\begin{equation*}
    \textrm{I} = \left \vert \frac{A \left ( \int_{N_r} h^2 -(\psi h)^2 \right)}{\left( B + \int_{N_r} h^2 \right) \left( B + \int_{N_r} (\psi h)^2 \right) } \right \vert, 
        \qquad
   \textrm{II} = \left \vert \frac{B \left ( \int_{N_r} |\nabla (\psi h)|^2 - |\nabla h|^2 \right)}{\left( B + \int_{N_r} h^2 \right) \left( B + \int_{N_r} (\psi h)^2 \right) } \right \vert. 
    \end{equation*}
    
\subsection{Controlling $\textrm{I}$}

We have that $A \leq \int_{\Omega_t} |\nabla h|^2$. Thus, by combining this term with the first one in the denominator we have that

\begin{equation}
    \label{eq:firstboundI}
\textrm{I} \leq \lambda_1   \frac{ \left ( \int_{N_r} h^2 -(\psi h)^2 \right)}{\left( B + \int_{N_r} (\psi h)^2 \right) } 
 \leq  \lambda_1   \frac{  \int_{N_r} h^2  }{ B }
\end{equation}
Since we have normalized so that $\|h\|_\infty =1$, (and the point where this is achieved is not near the neck $N_r$), the gradient estimate implies the bound
 \begin{equation}\label{L2 lower bound}
     B > \int_0^{2 \pi} \int_0^{\frac{1}{\|\nabla h\|_\infty}}(1- \|\nabla h\|_\infty \rho)^2 \rho \,d \rho \, d\theta = \frac{\pi}{6} \frac{1}{\|\nabla h\|^2_\infty} \geq C r^2,
 \end{equation}
where we have used the gradient estimate to observe that $\lambda_1 \leq \frac{C}{r^2}$.

Finally, using $L^\infty$ estimate through the neck, we have that
 \begin{equation}
 \label{eq:boundinneck}
  \int_{N_r} h^2  \leq 2 \underbrace{2 v(r)}_{\text{ width of $N_r$}}   \underbrace{(1+2 \delta) r \exp(2 v(r))}_{\text{ height of $N_r$}}  \underbrace{ C^2 \exp(- 2 c/r)}_{ C^0 \text{ bound on }h^2 }
 \end{equation}
where the initial factor of $2$ serves to account for how the metric deviates from a Euclidean one. 

Combining \eqref{eq:firstboundI}, \eqref{L2 lower bound}, and \eqref{eq:boundinneck}, we have that 
\begin{equation} \label{Final estimate on mathcal A}
    \textrm{I} \leq \frac{C}{r^3}  v(r) \exp(2 v(r)) \exp(-2c/r) ,
\end{equation}
which implies that
$ \lim_{r \to 0 } \textrm{I} = 0.$

\subsection{Controlling $\textrm{II}$}
We have 
    \begin{align*}
        \textrm{II} &\leq  \left \vert \frac{ \left ( \int_{N_r} |\nabla (\psi h)|^2 - |\nabla h|^2 \right)}{\left( B + \int_{N_r} h^2 \right) } \right \vert 
         \leq  \frac{ \int_{N_r} 2 h^2 |\nabla \psi|^2 +2  |\nabla h|^2 }{ \|h\|_2^2}. 
    \end{align*}
Using Inequality \eqref{L2 lower bound} and the gradient estimate, we have that

    \begin{align}
        \textrm{II} & \leq  \frac{2 \text{Vol}(N_r) \left( C^2 \exp(-2c/r)\frac{2^2}{v(r)^2} + 2 \frac{C}{r^2} \right) }{  C r^2 } \nonumber \\
        & \leq   \frac{4 r v(r) \left( C^2 \exp(-2c/r) \frac{2^2}{v(r)^2} + 2 \frac{C}{r^2} \right) }{  C r^2 }, \label{Final estimate on mathcal B}
    \end{align}
where the extra factor of $2$ accounts for the deviation from a Euclidean domain.  By taking $\delta < 2c,$ we can see that this expression goes to zero as $r \to 0.$

\subsection{A final continuity argument}

\label{subsec: Finishing the proof}

For most values of $t$, the function $\psi h$ will not be orthogonal to the principal eigenfunction $h$ on $\Omega_{r,L,t}$. However, by varying $t$ from $0$ to $L$,  the function 
\[ F(t) = \int_{\Omega_t} (\psi  h) \cdot h  \]
will depend continuously on $t$ by the continuity of solutions to linear equations in terms of their coefficients.  Lemma \ref{Needed for the continuity argument} forces $F(t)$ to be negative whenever $t< \frac{L}{9}$ and $I(t)$ to be positive whenever $t> \frac{8L}{9}$. Therefore, there must be a $t_0$ for which this integral vanishes. 
This completes the proof, since equations \eqref{Final estimate on mathcal A} and \eqref{Final estimate on mathcal B} show that the difference between the Rayleigh quotients of $\psi  h$ and $h$ is very small (and can be made arbitrarily small by letting $r$ go to zero). For $t=t_0$, $\lambda_2(\Omega_{r, t_0,L}) - \lambda_1(\Omega_{r,t_0,L}) \leq \mathcal R(\psi h) - \mathcal R(h)$, so the difference between the eigenvalues must also go to zero.

Before moving on, it is worth remarking about the quantitative estimates we have shown. More precisely, we have shown that we can find convex domains $\Omega$ of diameter $L$ and inscribed radius $r$ so that
\[ \lambda_2(\Omega)-\lambda_1(\Omega) < C e^{-cr^{-1}} \]
where $C$ and $c$ are constants depending on the curvature, its derivatives and $\exp(L/2)-1$. 

With sharper estimates, it is possible to refine the quantities $C$ and $c$. However, from a qualitative perspective, we expect the form of this expression to be sharp as this is the estimate in the hyperbolic case, where the estimates are qualitatively sharp.

 Using a more refined technique to define the domains, one can prove the same theorem under the weaker assumption that there exists a minimizing geodesic of length $D + \epsilon$ rather than $2D$, which we will do in Subsection \ref{A new continuity argument}. However, that argument is more technical, so we have used sliding domains here.

\section{Any negative curvature annihilates fundamental gaps} 
\label{sec:Three dimensional case}

We now turn our attention to constructing convex domains with arbitrarily small fundamental gaps in higher dimensions. The strategy of creating a neck is the same, although when the sign of the curvature is mixed this is more of a \emph{flattening} since we have extra dimensions. The estimates are substantially more delicate. For concreteness we only work in three-dimensions, which makes the book-keeping easier but does not simplify the argument in any substantial way.

From our hypothesis, there is a tangent plane of negative sectional curvature. We consider a geodesic $\gamma(x)$ of length $2L_0$ parametrized from $-L_0$ to $L_0$ so that $\gamma(0)$ is the base point of the aforementioned plane and $\gamma'(0) = e_1$ belong to the plane.   At every $\gamma(x)$, the quadratic form
    \begin{equation}
        \label{eq:Rquadratic}
    g( R(X,\dot \gamma) Y,  \dot \gamma) 
    \end{equation}
is symmetric. Let $\xi_2(x)$ and $\xi_3(x)$ be eigenvectors of \eqref{eq:Rquadratic} of length one with corresponding eigenvalues $\kappa_2(x)$ and $\kappa_3(x)$. For our Fermi coordinates, we choose $e_2 = \xi_2(0)$ and $e_3 = \xi_3(0)$ at $x=0$. We will assume that $\kappa_2(0)$ is negative, but make no assumptions on $\kappa_3(0)$. The proof is substantially harder when $\kappa_3(0)$ is positive. Let $(x,y,z)$ be Fermi normal coordinates centered around $\gamma$ and, for convenience, suppose that $y$ and $z$ are both bounded by $1$ in absolute value. 

Before we start discussing the proof, let us fix some notations: 
\begin{definition}
\label{def:curvature-bounds}
We denote by 
  \begin{align*}
  K_2 = \max_{x \in [-L_0, L_0]} \kappa_2(x),\quad 
  K_3 = \max_{x \in [-L_0,L_0]} |\kappa_3(x)|,
  \end{align*}
  and 
  \[
  K = \max_{x \in [-L_0, L_0]} \max_{X, Y \in T_{\gamma(x)} M} |\kappa(X,Y)|,
  \]
  where $\kappa(X,Y)$ is the sectional curvature. 
  \end{definition}

For now, let us choose $L_0$ so that
\begin{equation} \label{Curvature pinching along geodesic}
   \frac{9 \kappa_2(0)}{8} < k_2 \leq  K_2 
    < \frac{7 \kappa_2(0)}{8} < 0. 
\end{equation}
   The length $L\leq L_0$ will be determined later in the proof. 
In order to organize the geometric quantities involved, we start by labelling our directions.

\begin{definition}
The $x$ direction is the length, which ranges from $-L$ to $L$. The $y$ direction is the height. The $z_1, \ldots, z_{n-2}$ directions are the depth. In three dimensions, there is only a single depth direction.
\end{definition}

\subsection{The strategy}
\label{ssec:strategy}
We described the hurdles caused by the possible presence of positive curvature and how to overcome them. 
 

  \begin{enumerate}[label=(\roman{enumi})]
    \item Since we are studying convex domains, the problem comes down to understanding the behavior of geodesics in our space. Instead of studying geodesics directly, we can reduce the problem to analyzing Jacobi fields along $\gamma$. 
%
    More precisely,
    if $p$ and $q$ are two points which are at distance at most $r$ from $\gamma$, then the geodesic $\gamma_1$ between these two points can be parametrized as
      $ \gamma_1(x) = \exp_{\gamma(x)} (J(x)+O(r^2)),$ 
    where $J(x)$ is the Jacobi field whose boundary conditions are $\log_{\gamma(x(p))}(p)$ and $\log_{\gamma(x(q))}(q)$. 
%
    Here, the $O(r^2)$ error term can be bounded by the length of the geodesic and the $C^2$ norm of the curvature (as in the two-dimensional case). 
    For a geodesic of bounded length, we will always be able to choose $r$ small enough so that the dominant term in the analysis is the Jacobi field. Thus, if we can establish an estimate at the level of Jacobi fields, it will hold for geodesics.
    \item In our two-dimensional proof, we did not directly use the fact that the \emph{size} of the neck was small, but instead that the vertical cross-slices through the neck had a very large principal eigenvalue. We want to replicate the same phenomenon here. The main contribution to the first eigenvalue of the cross-slices should come from  the negative curvature direction (i.e. the height).  For this reason, the depth will be larger than the height and the ratio of depth to height will be determined by a large constant $\rho$. 
    \item In our chosen Fermi coordinates, the quadratic form \eqref{eq:Rquadratic} is not diagonal for all $x$. To keep track of the rotation, we define $\theta(x)$ to be the angle between $\xi_2(x)$ and $\frac{\partial}{\partial y}$. 
    By choosing $L$ small enough, $\theta$ will stay small, and the Jacobi field equations ``almost decouple". The length $L$ will depend on the derivative of the curvature tensor.
    \item The Jacobi fields satisfy a system of equations in coordinates. Because we control $\theta$, we can compare the solutions to our Jacobi equations to solutions to a nearby system where the equations are decoupled. We apply comparison arguments to show that given two points whose height is bigger than $1$, the height of the Jacobi field is uniformly strongly concave between these two points.
    \item We create a polyhedron in Fermi coordinates which roughly resembles an orthogonal parallelipiped whose ratio of depth to height is $\rho$. We then apply the previous strong concavity to show flattening in the height near the middle when the heights on the ends are big enough. In this construction, the precise polyhedron has a parameter $\alpha$ which slightly deforms it from being an orthogonal parallelipiped.
     Since the estimate for flattening is uniform in the $z$-coordinates (so long as they are bounded by $\rho r$ in absolute value), we get a domain as in Figure \ref{fig:A model domain in 3D}, which will have a neck where the principal eigenfunction must become very small.
     \item From this, we can repeat the proof as in the negatively pinched case until the final step, where we apply the continuity argument. It is not possible to use sliding domains as not all of the sliding domains will have a neck. Instead, we make use of the parameter $\alpha$ to find another family of domains which all have necks and where the neck moves from one side of the domain to the other.
\end{enumerate}

\subsection{Negatively pinched metrics}
\label{Higher dimensional case}

When the sectional curvature is negatively pinched, this strategy works to build domains whose fundamental gap is arbitrarily small. In this case, there is no need to bound $\rho$ or $\theta$, since the Jacobi field analysis is simpler and shows that all the Jacobi fields along the geodesic bend outward. And in the continuity argument, there is no need to take $\alpha$ very close to $1$. As a result, we can eliminate all of the restrictions on $L$ and so construct convex domains whose diameter is arbitrary (up to the diameter of the manifold) and whose fundamental gap is arbitrarily small.

\begin{figure}[htbp]
    \centering
    \includegraphics[width=.9\linewidth]{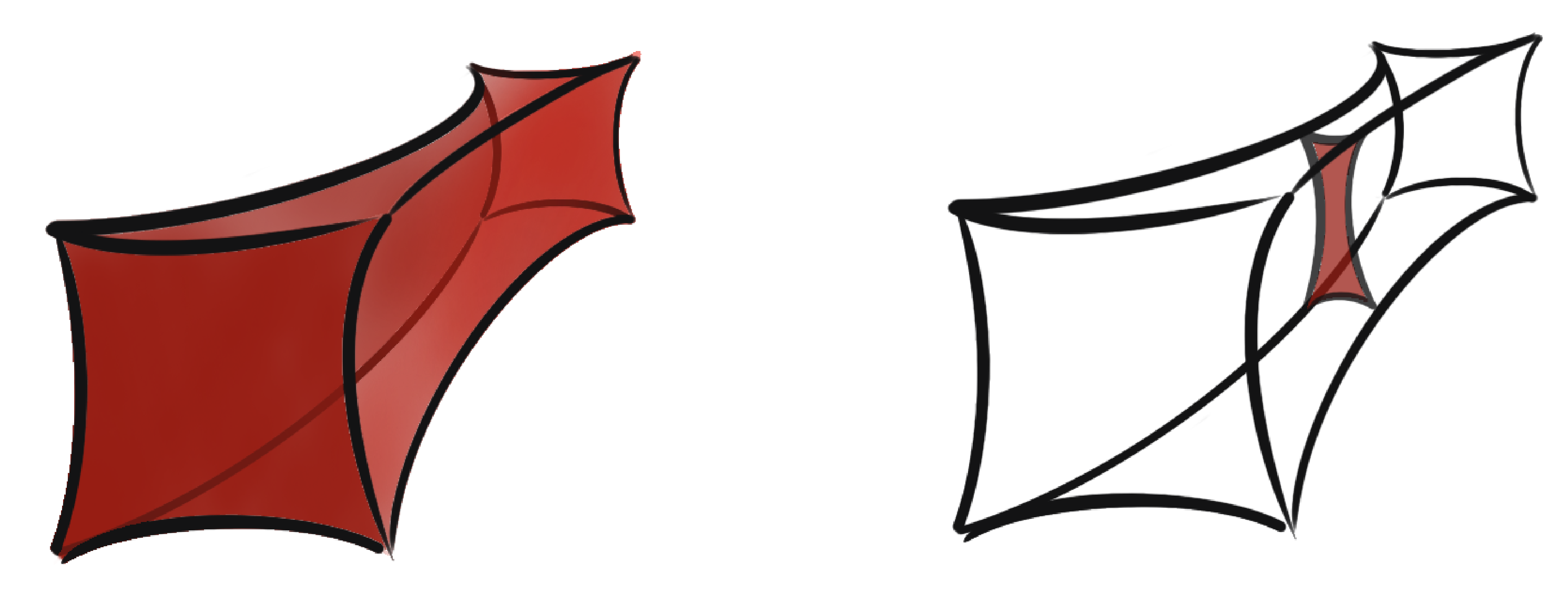}
    \caption{A domain in negatively curved three-dimensional space and the slices used to define $V(x)$}
    \label{fig:Negative curvature}
\end{figure}

\subsection{Determining the depth-to-height ratio $\rho$.}
In order to establish the flattening phenomenon, we will need the eigenvalue of the cross-slices to achieve their maximum in the middle of the domain. To do so, we pretend that the curvature tensor is constant along $\gamma(x)$ and allow the $y$-coordinate of the Jacobi fields to expand according to $\cosh(\kappa_2(0) x)$ and the $z$-coordinate of the Jacobi fields to contract according to $\cos(\kappa_3(0) x)$. For each $x$, the principal Dirichlet eigenvalue of the $x$-slice would be $\frac{\pi^2}{\cosh^2(\kappa_2(0) x)} + \frac{\pi^2}{\rho^2 \cos^2(\kappa_3(0) x)}$. Because the sectional curvatures change along $\gamma$, we need to allow some wiggle room so choose $\rho$ to be large enough to make the following function concave from $-L_0$ to $L_0$:
        \begin{equation} \label{Assumption on rho}
         \frac{\pi^2}{\cosh^2(K_2 x)} + 4 \frac{\pi^2}{\rho^2 \cos^2(K x)}.
        \end{equation} 
    This is one of the main points of the argument in two dimensions: in order to establish estimate for the second derivative in Lemma \ref{lem:second-derivative-V}, it was only necessary to prove that the eigenvalue of the cross-slices became large in the neck, not  that the height of the cross-slices was small. 
    
    Note that by a straightforward computation, at $x=0$ we have the identity
    \begin{equation*} \frac{d^2}{dx^2} \left( \frac{\pi^2}{\cosh^2(K_2 x)} + 4 \frac{\pi^2}{\rho^2 \cos^2(K x)}\right ) = \pi^2 \left( -2 K_2 + \frac{8}{\rho^2} K^2 \right),
    \end{equation*}
    which implies that 
    \begin{equation} \label{Lower bound on rho}
    \rho > 2 \frac{K}{|K_2|} >  2 \frac{K}{|\kappa_2(0)|} \geq 2.
    \end{equation}
We will use the fact that $\rho \geq 2$ throughout to simplify the argument.

\subsection{Length restrictions: Part I}
\label{ssec:geodesic-flattening}

Now that we have chosen $\rho$, we can choose a suitable length for the geodesic. 
First, we require that $L$ is small enough so that any solution to the boundary value problem
\begin{equation} \label{Conjugacy assumption}
    \ddot j = - K j \quad \quad j(-L) = j(L) = 1
\end{equation}
satisfies $1<j(t)<2$ for all $t \in (-L,L)$ where $K$, as before, is a bound on the sectional curvature in the tube domain. In other words, we assume that  $ \cos(K L)> \frac{1}{2}$. Intuitively, this assumption gives a qualitative way to say that the length is small relative to the conjugacy radius along the geodesic.

The second restriction on $L$ has been mentioned in \ref{ssec:strategy} and ensures that the rotation angle $\theta$ is small so the Jacobi fields can be approximated by ones for which $\theta \equiv 0$.   The key point here is that when the equations decouple (i.e. when $\theta\equiv 0$), the height of the Jacobi field evolves independently of its depth. 

\subsection{The Jacobi fields}
\label{ssec:JacobiFields}
In general, the Jacobi field equations do \emph{not} decouple, and instead we need to estimate the quantity
$
    \left \langle \frac{d^2}{d x^2} J, \frac{\partial}{\partial y} \right \rangle = \frac{d^2}{d x^2} J^y. 
$ 
  The equality is because $\frac{\pd}{\pd y}$ is parallel along $\gamma$ and we have a similarly property for the $z$-coordinate of the Jacobi field.

We use the conventions
\begin{align*}
\xi_2 &=  \cos(\theta) \frac{\partial}{\partial y} + \sin(\theta) \frac{\partial}{\partial z} \\
\xi_3 &=  -\sin(\theta) \frac{\partial}{\partial y} + \cos(\theta) \frac{\partial}{\partial z}.
\end{align*}
The angle $\theta$ depends on $x$, is zero at $x=0$, and is a differentiable function. 
Then, for a given $x$-value, we have that 
\begin{align*}
     \frac{d^2}{d x^2} J & =  - \kappa_2(x) \, \textrm{proj}_{\xi_2} J - \kappa_3(x) \, \textrm{proj}_{\xi_3} J  
      =  - \kappa_2(x) (J \cdot \xi_2) \xi_2 - \kappa_3(x) (J \cdot \xi_3) \xi_3 \\
     & =  \left( - \kappa_2(x) \cos(\theta) J^y - \kappa_3(x) \sin(\theta) J^z \right) \left( \cos(\theta) \frac{\partial}{\partial y} + \sin(\theta) \frac{\partial}{\partial z} \right) \\
     & \ \  + \left( \kappa_3(x) \sin(\theta) J^y - \kappa_3(x) \cos(\theta) J^z \right)  \left( -\sin(\theta) \frac{\partial}{\partial y} + \cos(\theta) \frac{\partial}{\partial z} \right) 
\end{align*} 
Collecting terms in the $\frac{\partial}{\partial y}$ direction and in the $\frac{\pd}{\pd z}$ direction, we have that
\begin{align}
\label{eq:Jy}
  \frac{d^2 }{dx^2}J^y  + (\kappa_2(x) \cos^2(\theta) + \kappa_3(x) \sin^2(\theta)) J^y &= - (\kappa_2(x)+\kappa_3(x)) \frac{\sin(2 \theta)}{2} J^z\\
\notag 
  \frac{d^2 }{dx^2}J^z   +( \kappa_2(x) \sin^2(\theta)  + \kappa_3(x) \cos^2(\theta)) J^z &= (\kappa_3(x)-\kappa_2(x)) \frac{\sin(2 \theta)}{2} J^y.
\end{align} 
From this, we see that we want to pick $L$ small enough so that whenever $|J^y|> \frac{1}{2}$ and $|J^z|< A$ for some constant $A$, we have that the coefficient $-\kappa_2(x) \cos^2(\theta)$ in equation \eqref{eq:Jy} dominates.  This implies several conditions on $L$, which are listed in Proposition \ref{prop:barriers-for-JF}.

%


\subsection{Geodesic flattening in the presence of negative curvature}

Now we consider points $p = \gamma(a)$, $q = \gamma(b)$ with $a, b \in (-L, L)$ and two vectors $V_p \in T_{p} M$ and $V_q \in T_{q} M$ which are both perpendicular to $\dot \gamma$. To give some intuition for why to consider these quantities, these vectors are scaled logarithms for points in the convex domain we will construct. Furthermore, the Jacobi field connecting the vectors approximate the geodesic between these points. For now, let us derive the relevant estimates and then apply them to building the convex domain.

\begin{proposition}[Flattening of Jacobi fields]
\label{prop:barriers-for-JF}
We suppose that $\rho>2$ is a given constant and the vectors $V_p$ and $V_q$ satisfy the estimates
$V_p^y, V_q^y > 3/4$ and $|V_p|< \rho \cos(K a) $ and $|V_q| <  \rho \cos(K b) $. 
If $L$ satisfies 
  \begin{enumerate}[label=(\roman{enumi})]
      \item \label{item:cos-bound}$\cos (KL)>\frac{1}{2}$, 
      \item \label{item:theta-bound} $\sin(\theta(x)) < \frac{|\kappa_2(0)|}{16 K \rho}$, for $x \in [-L,L]$, and 
      \item \label{item:coshL} $\cosh\left(\sqrt{\frac{-3 \kappa_2(0)}{2}} L \right) \leq \frac{3}{2}$, 
  \end{enumerate}
the Jacobi field $J(x)$ connecting $V_p$ and $V_q$ satisfies the estimate 
\begin{equation*}
    j_*(x) <  J^y(x) < j^*(x), \quad x \in (a,b)
\end{equation*}
where $j^*(x)$ is a function satisfying
\begin{equation*}
    \frac{d^2}{dx^2} j^* = -\frac{\kappa_2(0)}{2} j^* \quad j^*(a) = V_p^y, j^*(b) = V_q^y. 
\end{equation*} 
and $j_*(x)$ is a function satisfying
\begin{equation*}
    \frac{d^2}{dx^2} j_* = -\frac{3 \kappa_2(0)}{2} j_* \quad j_*(a) = V_p^y, j_*(b) = V_q^y. 
\end{equation*} 
\end{proposition}




We first recall a comparison lemma which will be used several times in the remainder of the argument.

%
%


\begin{lemma} \label{Concave envelopes}
Suppose $\varphi$ is a real solution on $(a,b)$ of 
  \[
\ddot{\varphi} = -g_1(x) \varphi 
 \]
and $\psi$ a real solution on $(a,b)$ of 
  \[
  \ddot{\psi} = -K \psi.
  \]
  Let $K > g_1(x)>0$ on $(a,b)$. If $\psi(a)= \varphi(a)>0$ and $\psi(b) = \varphi(b) >0$ and $b-a < \frac{\pi}{K}$
  then $\varphi (x) < \psi(x)$ on $(a,b)$.
\end{lemma}

\begin{proof}
By solving the ODE for $\psi$ explicitly, we find that the bound on $b-a$ ensures that there is a  solution $\psi$ that is positive on $(a,b)$.  By the Sturm comparison theorem, we see that $\varphi$ must also be strictly positive on $(a,b)$ (as there must be a root of $\psi$ between consecutive roots of $\varphi$). Since both are strictly positive, the function $\frac{\psi}{\varphi}$ is defined. It satisfies the equation
    \begin{align*}
    \ddot{\left( \frac{\psi}{\varphi}\right)} &= \frac{(\varphi\ddot{\psi} - \psi \ddot{\varphi})}{\varphi^2} - \frac{(\varphi\dot{\psi} - \psi\dot{\varphi}) 2 \dot{\varphi} }{\varphi^3} \\
    & = (g_2 - g_1) \frac{\psi}{\varphi} - \dot{\left(\frac{\psi}{\varphi}\right)} \frac{2 \dot \varphi}{\varphi^2}.
    \end{align*} 
Thus $\frac{\psi}{\varphi}$ can not achieve a positive minimum in $(a,b)$. This implies the result.  \end{proof}

\begin{proof}[Proof of Proposition \ref{prop:barriers-for-JF}]

 This will take two steps. Firstly, we will show that $J^z$ does not become too large. Then we will show that $j_{*}$ is a subsolution for \eqref{eq:Jy}. Similarly, $j^*$ is a supersolution and combining these observations gives the result.  
 
{\bf Bounding $|J^z|$.} Instead of dealing with $J^z$, we work with $|J(x)|$, which will automatically give us the estimate we want and will be easier to handle.\footnote{This estimate could also be done by comparing the space to one of very positive curvature using the Rauch comparison theorem.}  Let us consider the Jacobi field equations, which are
\begin{equation*}
\frac{d^{2}}{d x^{2}} J(x)+R(J(x), \dot{\gamma}(x)) \dot{\gamma}(x)=0.
\end{equation*}
We decompose the second term into the component which is parallel to $J(x)$ and the component which is perpendicular to $J(x)$ to get  
\begin{equation*}
\frac{d^{2}}{d x^{2}} J(x)+\kappa(J(x), \dot{\gamma}(x)) J(x) + (\mathcal{R} \ast J) J^\perp(x) =0,
\end{equation*}
where the second term is an algebraic combination of curvatures $\mathcal{R}$ and components of the Jacobi field $J(x)$ and $J^\perp(x)$ is a unit vector which is perpendicular to $J(x)$ (and will depend on $x$ in general).

Computing the evolution of $|J(x)|$, we see the term induced by $J^\perp(x)$ plays no role, so we have the estimate
$ \frac{d^{2}}{d x^{2}}  |J(x)| = - \kappa(J(x), \dot{\gamma}(x)) |J(x)|. $
We can then apply Lemma \ref{Concave envelopes} to bound the size of $|J(x)|$, and our assumption on $L$ ensures that $|J(x)|$ (and hence $|J^z(x)|$) has size at most $\rho \cos(K x)$.

{\bf The function $j_*$ is a lower barrier.}
Let us consider $j_{**}$, the solution to $\frac{d^2}{dx^2} j_{**} = - \frac{3 \kappa_2(0)}{2} j_{**}$ with boundary condition $j_{**}(a) = j_{**}(b) = 3/4$. It can be written explicitly and condition \ref{item:coshL} implies $j_{**}(x) \geq 1/2$ for every $x \in (a, b)$. Therefore $j_{*}(x) \geq 1/2$. 

Applying the left-hand operator of \eqref{eq:Jy} to $j_*$, we have 
  \begin{multline*}
   \frac{d^2 }{dx^2}j_*  + (\kappa_2(x) \cos^2(\theta) + \kappa_3(x) \sin^2(\theta)) j_*\\
   =\left(- \frac{3 \kappa_2(0)}{2} +  \kappa_2(x) \cos^2(\theta) + \kappa_3(x) \sin^2(\theta)\right) j_* =: c(x) j_*
  \end{multline*}
and the coefficient $c(x)$ in front of $j_*$ can be bounded from below. Indeed,  the bounds \eqref{Curvature pinching along geodesic} and condition \ref{item:theta-bound} give
  \begin{align*}
      c(x)
      & \geq \frac{3 |\kappa_2(0)|}{2} - |\kappa_2(x)| - K \sin^2(\theta) \\
      &
      \geq \frac{3 |\kappa_2(0)|}{2} - \frac{9 |\kappa_2(0)|}{8} - K \frac{|\kappa_2(0)|^2}{32^2 K^2 \rho^2 }\\
      & \geq \frac{3 |\kappa_2(0)|}{8} - \frac{|\kappa_2(0)|}{32^2 \rho^2}
       \geq \frac{|\kappa_2(0)|}{4},
  \end{align*}
  where we used \eqref{Lower bound on rho} in the last inequality. The nonhomogeneous term of \eqref{eq:Jy} is controlled using the fact that $|J^z| \leq \rho$
    \begin{equation}
    \label{eq:nonhomogeneous}
        \left|(\kappa_2(x) + \kappa_3(x)) \frac{\sin(2\theta)}{2} J^z \right|\leq 2 K |\sin(\theta)| \rho \leq \frac{|\kappa_2(0)| }{16}. 
    \end{equation}
  Because $j_* \geq 1/2$, it is a subsolution of \eqref{eq:Jy}. Therefore $j_{\ast} (x) \leq J^y (x)$ for every $x\in (a,b)$. 
 
  {\bf The function $j^*$ is an upper barrier.} Similarly, one can show that $ \frac{d^2 }{dx^2}j^*  + (\kappa_2(x) \cos^2(\theta) + \kappa_3(x) \sin^2(\theta)) j^* \leq - \frac{|\kappa_2(0)|}{4}$. Thanks to \eqref{eq:nonhomogeneous} and the observation that $j^*(x) \geq j_*(x) \geq 1/2$, we have that $j^*$ is a supersolution of \eqref{eq:Jy}. This concludes the proof. 
  \end{proof}

\subsection{Length restrictions: Part II}
\label{2 length 2 restrictions}

Now that we have established the upper and lower barriers $j^*(x)$ and $j_*(x)$, respectively, we must impose one more condition on the length, which will play a role at the very end when we apply the continuity argument. For reasons that will become clear in Subsection \ref{A new continuity argument}, we want the upper barrier to be fairly shallow, which will allow a small perturbation of the endpoints to change the upper barrier from increasing to decreasing on the interval $[-L,L]$. To make this precise, we impose a final length restriction on $L$
\begin{equation} \label{Upper convex envelope length restriction}
    \sinh^2 \left( \sqrt{\frac{-\kappa_2(0)}{2} } L\right)  < \frac{1}{20}
\end{equation}  

\subsection{Building the domain}

 The key to making the construction of the domain work is that the barriers are \emph{uniform} in the depth of $V_p$ and $V_q$ so long as they are both shallower than $2 \rho$.
 
We now consider 8 points arranged in $M$ whose $(x,y,z)$ coordinates are 
\begin{equation*} \label{Definition of alpha}
    \left ( -L, \pm \alpha  r, \pm \rho r \right ) \qquad   \left ( L, \pm r, \pm \rho r \right ).
\end{equation*}
 These points form the vertices of a parallelipiped in the Fermi coordinates, and we consider the convex hull of these points. We call this domain $\Omega_{\alpha,r}$. For $r$ small enough and $\alpha$ close enough to $1$, we want to show that this domain has a neck. At the end of the proof, we will specify a value for $\alpha$, which will replace the role of $t$ in the original continuity argument. For now, we will only insist that $\alpha \in \left( \frac{10}{11},\frac{11}{10} \right).$

Let us now consider the height of the domain, which is defined to be
\begin{equation*}
   \texttt{H}(x,z) = \sup_{y} \left\{ (x,y,z) \in \Omega_r \right \}.
\end{equation*}

In other words, the height is the maximal $y$ value for a fixed $x$ and $z$ value.
We call the collection of points which attain the height the ``top" of the domain. In two dimension, the height was achieved by a geodesic and we could use the Rauch comparison theorem to control the geometry of the domain. However, in higher dimensions the top can be much more complicated and in general is not smooth.
The estimates on the Jacobi fields were obtained uniformly in the $z$ coordinates, so they hold no matter which piece of geodesic realizes the top of the domain.

\begin{figure}
    \centering
    \includegraphics[width=.9\linewidth]{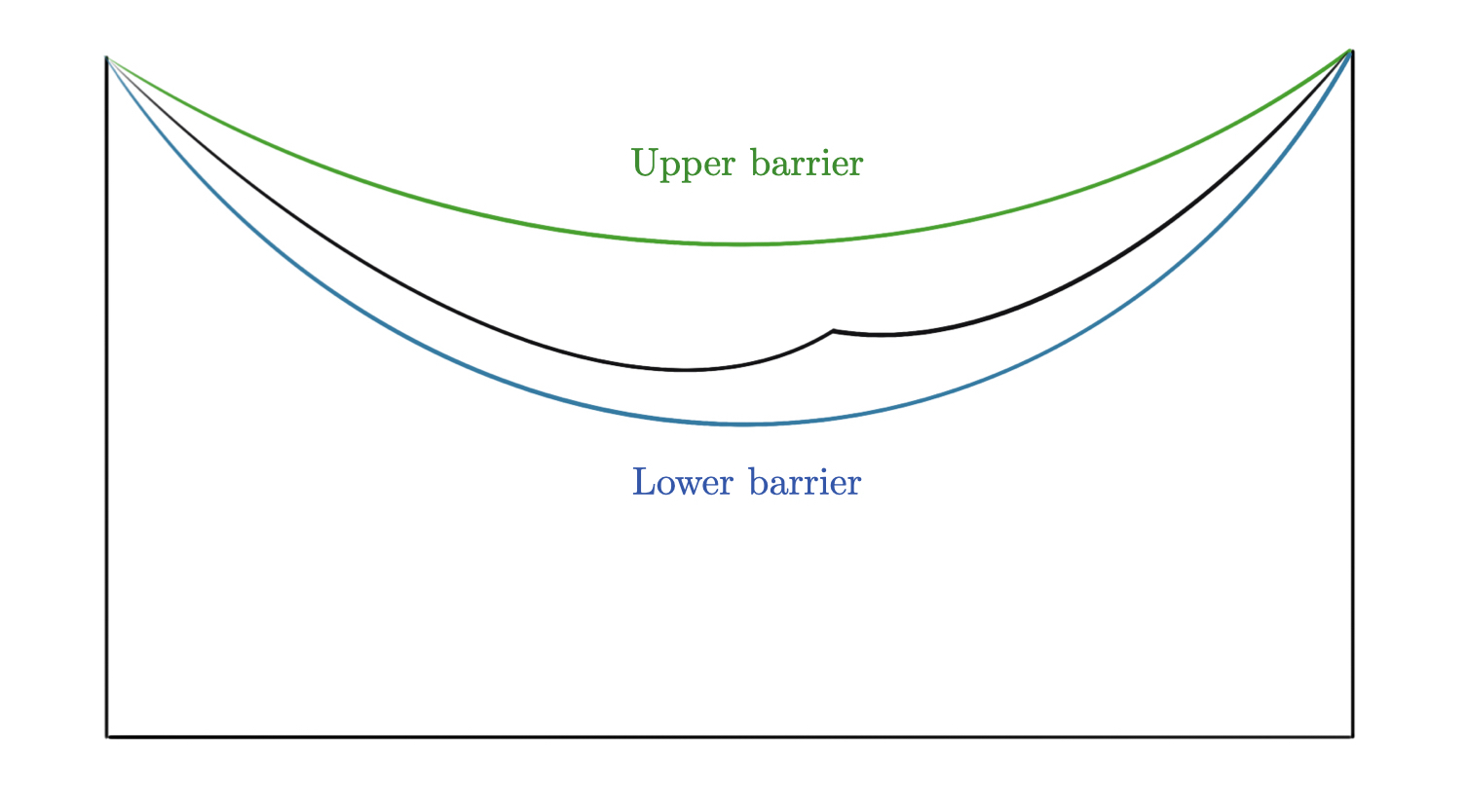}
    \caption{The height of the domain\protect \footnotemark as a function of $x$ (for a fixed $z$ value) and the barriers bounding it}
    \label{fig:Convex envelopes}
\end{figure}

\footnotetext{In this figure, we have purposely drawn the height to be neither smooth or convex. Using more careful analysis, it is possible to show that the Hessian of the height function can be bounded from below using a lower bound on the sectional curvature, but in general we can expect it to have corners.}

From here, it is possible to replicate the rest of the argument until the final step involving the continuity argument, which requires using $\alpha$ instead of $t$.

\subsection{A new continuity argument}

\label{A new continuity argument}


In the case of mixed curvature, we cannot use our original continuity argument using sliding domains. As shown in Figure \ref{fig:Convex envelopes}, the height of the domain need not be convex in $x$. Therefore, if there is a gap between the upper and lower barriers, it is possible that when we try to slide the domain, for intermediate values of $t$ the domain $\Omega_{r,t}$ might achieve its maximal height near $\{x=0\}$, ruining the neck effect.

To get around this issue, we consider a different family of domains, which are parametrized by $\alpha.$ Changing $\alpha$ acts to change the height at one side of the domain while leaving the height at the other side fixed. The key improvement on this family compared to the sliding domains is that the convex upper and lower barriers are equal to each other at the endpoints for all $\alpha$. Therefore, even when the height has a local maximum in the interior of $\Omega$, this value cannot exceed the height at one or both of the endpoints.

In order to make this precise, let us state a brief lemma which follows from the properties of second order ODEs.

\begin{lemma} \label{Lipschitzness of the envelopes}
For all $\alpha \in \left(\frac{10}{11},\frac{11}{10} \right)$ and $J^z$ values with $|J^z|< 2 \rho$, both the upper and lower barriers are uniformly Lipschitz in $x$ with the constant only depending on the height at the endpoints and the bound on the curvature $K$.
\end{lemma}

We define the ``neck" of the domain to be centered at the $x$-value where the convex upper barrier attains its minimum.  Unlike in the previous case, the neck need not contain $x= 0$, and will in fact move from $x=-L$ to $x=L$ as $\alpha$ increases. We call the minimum value of the convex upper barrier the ``neck bound," and denote it by $\mathcal{B}$. We then define the neck to be the set where the convex upper barrier is less than $\mathcal{B}+\frac{1-\mathcal{B}}{4}$ or $\mathcal{B} +\frac{\alpha-\mathcal{B}}{4}$, whichever is larger. Because the upper barrier is convex, the neck is a single connected set and its width is bounded from below by Lemma \ref{Lipschitzness of the envelopes}. The particular $x$-value where the height is minimized depends on the $z$-coordinate, but since $z$ is bounded by $2 \rho r$, and the convex upper barriers are uniformly convex, changing $z$ will affect the argmin of $\texttt{H}( \cdot ,z)$ by $O(r)$ at most. For $r$ sufficiently small, this allows us to define the location and height of the neck consistently up to $O(r)$ (which can be discarded in the analysis).

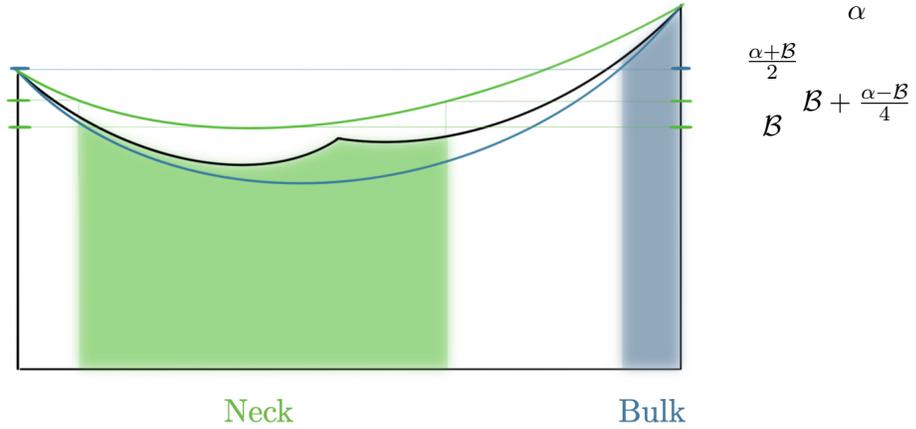
\begin{figure}[htbp]
\begin{tikzpicture}[scale=\MyScale]
\node at (-100,0)     {\includegraphics[width=.9\linewidth]{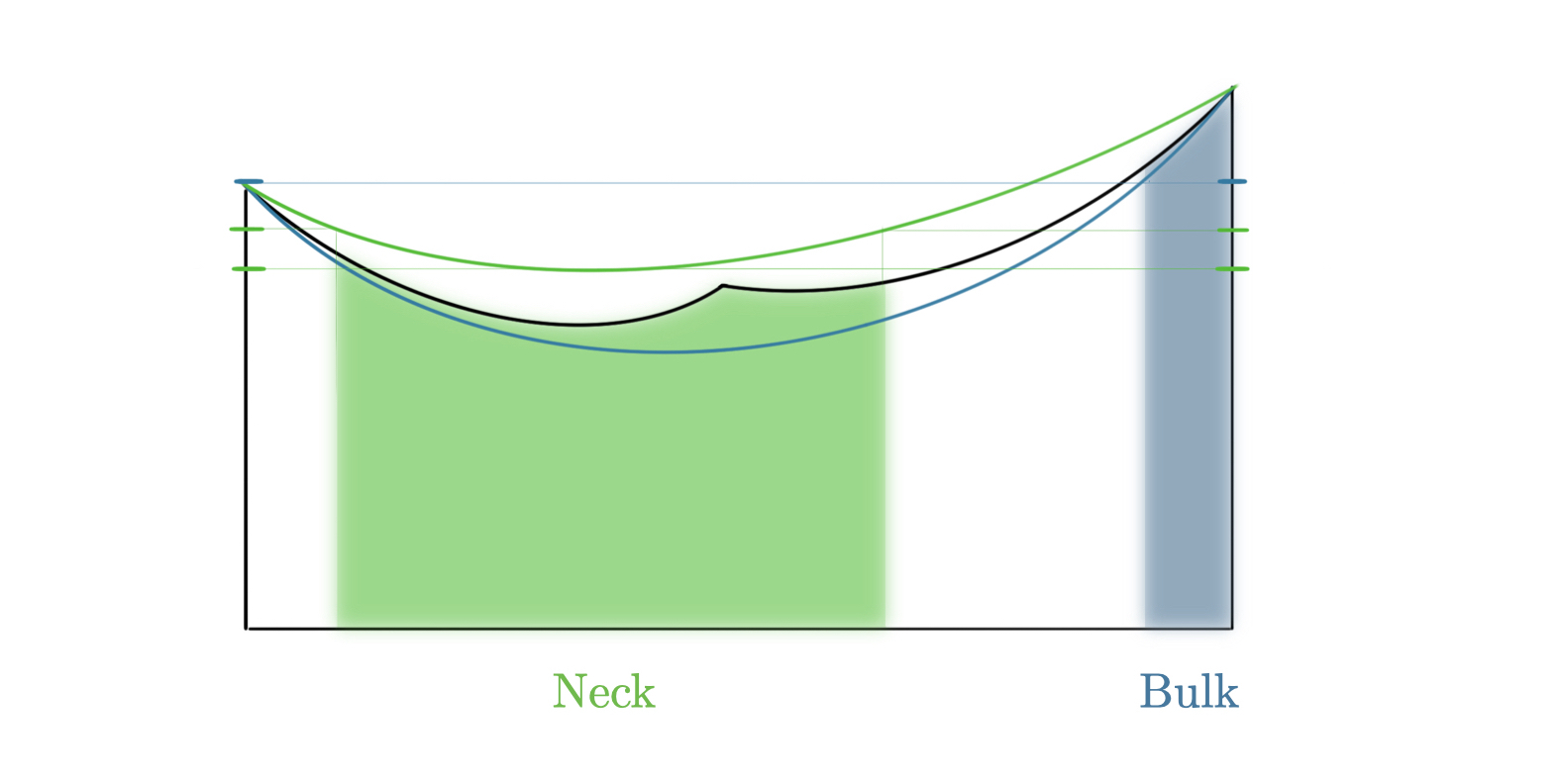}};
\node at (-96.5,.75) {$\mathcal{B}$};
\node at (-95.75,0.95) {$\mathcal{B} + \frac{\alpha-\mathcal{B}}{4}$};
\node at (-96.5,1.32) {$\frac{\alpha+\mathcal{B}}{2}$};
\node at (-95.75,1.75) {$\alpha$};
\end{tikzpicture}
    \caption{The neck and bulk of the domain\protect \footnotemark}
    \label{fig:neckandbulk}
\end{figure}

\footnotetext{In this figure, we have blurred the boundaries of the neck and bulk to emphasize that these are only defined up to a term of order $O(r)$. At first, it might seem counter-intuitive that the neck is much longer than the bulk, but as the height becomes very small (i.e., $r$ becomes small), the dominant term in determining the eigenvalue of domain is the height in the bulk.}

We then define the \emph{bulk} of the domain is the set where the convex lower barrier is larger than $\frac{1+\mathcal{B}}{2}$ or $\frac{\alpha+\mathcal{B}}{2}$, whichever is larger. Again using Lemma \ref{Lipschitzness of the envelopes}, we have the following observation.

\begin{proposition}
For all $\alpha \in \left(\frac{10}{11},\frac{11}{10} \right)$, the measure of the $x$-values in the bulk has a uniform lower bound $b$, which depends on $L$ and $K$ but is \emph{independent} of $r$.
\end{proposition}

Since the lower barrier is convex, the bulk has at most two connected components. Therefore, the bulk of the domain contains a rectangle whose dimensions are at least \[\frac{b}{2} \times \frac{ \max_{}\{1,\alpha\}+\mathcal{B}}{2} r \times \rho r \]
where $b$ is some constant which is smaller than $L$ but has a uniform lower bound as $r$ goes to $0$.
So for $r$ small enough, the first eigenvalue of the $\Omega_{\alpha,r}$ is determined by the height in the bulk, which means that it satisfies
\begin{equation*}
    \lambda_1(\Omega_{\alpha,r}) \leq \frac{4 \pi^2}{b^2} +  \frac{4 \pi^2}{\left(\max_{}\{1,\alpha\}+\mathcal{B}\right)^2 r^2} + \frac{\pi^2}{\rho^2 r^2} + O\left(\frac{1}{r}\right).
\end{equation*}

Note that the ratio of the maximum height of the neck (i.e., $\mathcal{B} +\frac{\max_{} \{\alpha,1\}-\mathcal{B}}{4})$ over the minimum height of the bulk (i.e., $\frac{\max_{}\{\alpha,1\}+\mathcal{B}}{2}$) is strictly less than 1 and this bound is uniform in $\alpha$ for $\alpha \in \left(\frac{10}{11},\frac{11}{10} \right)$. 

As before, we define $V(x)$ to be the $L^2$ integral
\begin{equation*}
    V(x_0) = \int_{\Omega \cap \{x = x_0\}} g^{xx} h^2 \, dy \,dz
\end{equation*}
and compute $\frac{\partial^2}{\partial x^2} V(x)$. At this point, we might worry that because the $x$-cross-slices need not be smooth sets, that the second derivative of $V$ may not exist. However, since $h^2$ vanishes to second order on the boundary, the problem terms coming from how the shape of the boundary changes will vanish when we calculate the second derivative of $V$.

 Using our bounds on the eigenvalues, within the neck we will obtain a bound on $\frac{\partial^2}{\partial x^2} V(x)$ which is very large (at least $\frac{C}{r^2} V$). By integrating this differential inequality, we are able to repeat the doubling estimate.

\begin{obs}
The principal eigenfunction must be exponentially small in the neck.
\end{obs}

When $\alpha = 1$, the bulk exists on both sides of the cross-slice $\{x=0\}$. On the other hand, for $\alpha$ sufficiently large so that the upper barrier is increasing, the bulk lies on the right side of the domain and includes $x=L$. Conversely, if the upper barrier is decreasing, the bulk lies on the left side and includes $x=-L$. From this, we find the following.

\begin{obs}
As $\alpha$ increases, the neck moves from the left of the domain to the right.
\end{obs}

As before, let $\psi$ be a smooth function that transitions rapidly from $+1$ to $-1$ in the neck region. All that is left to do is show that the integral
 \begin{equation} \label{Orthogonality test}
     F(\alpha,r) = \int_{\Omega_{\alpha,r}}  (\psi h ) \cdot h
 \end{equation}
switches signs when $\alpha$ is close enough to $1$ so that the domain is not too tall or short on either side.

To see this, observe that if the upper barrier is increasing, the neck is on the left side of the domain and so integral \eqref{Orthogonality test} is negative. On the other hand, if the upper barrier is decreasing, the integral is positive. However, by solving the relevant boundary value problem, Condition \eqref{Upper convex envelope length restriction} implies that the  upper barrier is monotonic in $x$ as soon as $|\alpha-1|>1/10$. This value is small enough so that all of the preceding estimates on the barriers and the doubling estimates still hold, which completes the proof.

\section{Acknowledgements} 

The authors would like to thank Guofang Wei and Malik Tuerkoen for their helpful comments. The first named author is partially supported by Simons Collaboration Grant 849022 (``K\"ahler-Ricci flow and optimal transport") and the second named author is partially supported by Simons Collaboration Grant 579756.

\appendix
\section{Estimates on the metric}

\label{ap:estimates-metric}
%
%

The Greek indices range from $2$ to $n$. 
The Roman indices range from $1$ to $n$.

\begin{lemma}
\label{lem:metric-estimate}
Suppose that in a tubular neighborhood of $\gamma$, we have that the sectional curvature is bounded between two constants $K_1 \leq \kappa \leq K_2$. Then there is a potentially smaller neighborhood of $\gamma$ for which the following estimate on the Riemannian metric expressed in Fermi normal coordinates holds:
\begin{gather}
\label{eq:metric-space}
g_{\alpha\beta}(x,y)= \delta_{\alpha\beta}+\frac{1}{3} R_{\eta\alpha\beta\nu}(x,0) y^\eta y^\nu + O(|y|^3) \\
\label{eq:metric-time}
g_{11}(x,y)= 1 +  R_{\eta 11 \nu} (x,0) y^\eta y^\nu + O(|y|^3) \\
\label{eq:metric-mixed}
g_{1\alpha}(x,y) = \frac{2}{3} R_{\eta 1 \alpha \nu}(x,0) y^{\eta} y^{\nu} + O(|y|^3)
\end{gather}
where $O(|y|^3)$ is a function $f$ for which $\lim_{|y| \to 0} \frac{|f(x, y)|}{|y|^3} = C$ (and $f$ has bounded derivatives in $y$'s and $x$ because everything is smooth).
\end{lemma}
From the Taylor expansion \eqref{eq:metric-space}, \eqref{eq:metric-time}, and \eqref{eq:metric-mixed}, we immediately get \eqref{eq:metric-estimates}.


\begin{proof}[Proof of Lemma \ref{lem:metric-estimate} ]
Recall the construction of our Fermi coordinates. We start with a point $Z$, and orthonormal frame $\{e_i\}_{i=1}^n$ at $Z$ and a geodesic $\gamma(x)$, $x\in(-\epsilon, \epsilon)$ with $\gamma(0)=p$.  

We extend our frame $\{e_i\}_{i=1}^n$ at $p$ to points $\gamma(x)$ by parallel transport.  In a tubular neighborhood of $\gamma(x)$, the chart $\phi: (-\epsilon, \epsilon) \times B_{R}^{n-1}$ for some small $R$ given by  
  \[
  \phi(x,y)= \exp_{\gamma(x)} (  y^\alpha e_\alpha), \quad \alpha=2, \ldots, n
  \]
defines Fermi coordinates. 

Let $E_{\alpha}(x,y) := \partial_{\alpha} |_{(x,y)}$ be the coordinate vector fields at $(x,y)$ and let $E_1 (x,y)$ be the Jacobi vector field along the geodesic $\tilde \sigma (s) :=\phi(x,sy)$ with initial conditions $E_1(x,0)=\gamma'(x)$ and $\partial_s E_1(x,0) = 0$. 

For $x$ fixed, $y \mapsto \phi(x,y)$ are normal coordinates on the $(n-1)$-dimensional manifold formed by geodesics rays perpendicular to $\gamma$ at $\gamma(x)$. Classical derivations give that the Taylor expansion of the metric in this coordinate system $\alpha=2,\ldots,n$ is \eqref{eq:metric-space} (see Sternberg pp. 225-227 \cite{sternberg1999lectures}). 
Because of the parallel transport of our frame $\{e_{i}\}$ along $\gamma$, we have $\Gamma_{ij}^k (x, 0) = 0$ for $i,j,k = 1, \ldots, n$. This is equivalent to $\nabla_k g_{ij}(x,0)=0$. Because the Christoffel symbols vanish for all $(x, 0)$, $\nabla_1\Gamma_{ij}^k (x,0) =0$, where the subscript $,1$ indicates (covariant) differentiation in the $x$ direction. 

For \eqref{eq:metric-time} and \eqref{eq:metric-mixed}, notice that $g_{1j}(x,0) =\delta_{1j}$ from our choice of coordinates. We now compute the second derivatives of $g_{11}$ and $g_{1\alpha}$ at $(x,0)$.  Let us keep $x$ fixed and for $w=w^{\alpha}e_{\alpha}$ a unit vector consider the geodesic $\tilde \sigma(s)$ given by $(x, sw)$ in coordinates.  Let $\langle \cdot,\cdot \rangle$ be the inner product of the metric of $M$ and let $'$ denote the derivative with respect to $s$ (the derivative $\frac{D}{ds}$ or $\nabla_{\tilde \sigma'}$), we have
  \begin{align}
\label{eq:g11} g_{11} &= \langle E_1, E_1 \rangle\\
 \label{eq:g11-prime} g_{11}' &= 2 \langle  E_1', E_1 \rangle \\
  \notag g_{11}'' &= 2 \langle E_1'', E_1 \rangle + 2 \langle E_1',  E_1'\rangle\\
    \label{eq:g11-second}&=- 2 \langle R(\tilde \sigma', E_1) \tilde \sigma', E_1 \rangle + 2 \langle E_1', E_1' \rangle
  \end{align} 
Setting $s=0$, the first derivative  and the second term of the second derivative vanish. We get the Taylor expansion
\[
g_{11}(\tilde \sigma(s)) = 1 - \langle R(w, e_1) w, e_1\rangle s^2 + O(s^3)
\]
Setting $y = sw$, we have \eqref{eq:metric-time}.

Let $Y_{\alpha}$ be the Jacobi field along $\tilde \sigma$ with initial conditions $Y_{\alpha}(0) = 0$ and $Y_{\alpha}'(0)=e_{\alpha}$. We do this because $E_{\alpha}$ is not a Jacobi field, however, we have that $ E_{\alpha} (\tilde \sigma(s)) = \partial_{\alpha} (\tilde \sigma) (s)) = s^{-1} Y_{\alpha}(s)$ (see Sternberg p.227). We have $s g_{1\alpha} = \langle E_1, Y_{\alpha}\rangle$ and we compute the derivatives the latter  
\begin{align*}
    \langle E_1, Y_{\alpha}\rangle'&= \langle E_1 ', Y_{\alpha} \rangle + \langle E_1, Y_{\alpha}'\rangle\\
    \langle E_1, Y_{\alpha}\rangle''& = \langle E_1'', Y_{\alpha} \rangle + 2 \langle E_1 ', Y_{\alpha}'\rangle + \langle E_1, Y_{\alpha}''\rangle\\
     & = - \langle R(\sigma', E_1) \sigma', Y_{\alpha} \rangle + 2\langle E_1 ' , Y_{\alpha}' \rangle - \langle E_1, R(\sigma', Y_{\alpha}) \sigma' \rangle \\
   \langle E_1, Y_{\alpha}\rangle'''  & = - \langle (R(\sigma', E_1) \sigma')', Y_{\alpha} \rangle - \langle R(\sigma', E_1) \sigma', Y'_{\alpha} \rangle+ 2\langle E_1 '' , Y_{\alpha}' \rangle  + 2\langle E_1 ' , Y_{\alpha}'' \rangle \\
   & \ \ \  -\langle E_1', R(\sigma', Y_{\alpha}) \sigma'\rangle -\langle E_1, (R(\sigma', Y_{\alpha}) \sigma')'\rangle 
\end{align*}
Setting $s=0$ and recalling that $E_1'(t,0)$, $Y_{\alpha}(t,0) =0$, we get that $\langle e_1, Y_{\alpha} \rangle$, its first and second derivatives all vanish. For the third derivative, the only part of $(R(\sigma', Y_{\alpha}) \sigma')'$ that is not zero at $(t,0)$ is $R(\sigma', Y_{\alpha}') \sigma'$ so we have 
\[
\langle E_1, Y_{\alpha}\rangle'''(t,0) = -  4 \langle R(w, e_1)w, e_{\alpha} \rangle.
\]
Thus 
\[
s g_{1,\alpha}(\tilde \sigma(s)) = - \frac{2}{3} \langle R(w, e_1)w, e_{\alpha}  \rangle s^3 + O(s^4)
\]
Taking $y=sw$ and dividing by $s$ give \eqref{eq:metric-mixed}. 
The derivation of \eqref{eq:metric-space} is a similar computation with $s^2 g_{\alpha\beta} = \langle Y_{\alpha}, Y_{\beta}\rangle$. As mentioned above, it is done in \cite{sternberg1999lectures}.
\end{proof}

\begin{proof}[Proof of Proposition \ref{prop:estimates-metric}]
First note that we have bounds on $|E_i|$ for all $i$'s thanks to Lemma \ref{lem:metric-estimate}. 

The bounds \eqref{eq:gamma-estimates} for $i=2,\ldots, n$ follow directly from the computations above. For $g_{11}$, equations \eqref{eq:g11-prime} and \eqref{eq:g11-second} at $s=0$ are equivalent to  
    \begin{align*}
    \nabla_V g_{11}(x,0) &=0\\
    \nabla^2_{VV} g_{11}(x,0) & = - 2 \langle R(V, e_1)V, e_1\rangle
    \end{align*}
for any vector $V = \sum_{\alpha=2}^n V^{\alpha} e_{\alpha}$ perpendicular to $E_1$. Setting $V = E_{\alpha} + E_{\beta}$ in the second derivative and using the linearity of the connection, we get that $\nabla^2_{VV} =  \nabla^2_{E_{\alpha}E_{\alpha}} + \nabla^2_{E_{\alpha}E_{\beta}} + \nabla^2_{E_{\beta}E_{\alpha}} + \nabla^2_{E_{\beta}E_{\beta}}$. Thus one can bound the second mixed derivatives $\nabla^2_{\alpha\beta} g_{11}$ at $(x,0)$ with terms involving only curvature. Since $\nabla_\beta g_{11}$ vanishes at $(x,0)$, the Taylor expansion of $\nabla_\beta g_{11}$ gives 
\[
|\nabla_\beta g_{11} (x, y)| \leq C |y| 
\]
where $C$ depends on curvature bounds. The computations are similar for $\nabla_{\alpha} g_{\beta\eta}$ and $\nabla_{\alpha} g_{1\beta}$. 

For the Taylor expansion of $\nabla_1 g_{11}$, we differentiate \eqref{eq:g11}, \eqref{eq:g11-prime}, and \eqref{eq:g11-second} with respect to $\nabla_1$. We already noted that $\nabla_1 g_{11}(x,0) =0$. Also remark that $\nabla_1 \nabla_{\alpha} g_{11} = \nabla_{\alpha} \nabla_1 g_11$. The Christoffel symbols vanish at $(x,0)$ so the covariant derivatives can be ordinary derivatives and we had that $g_{11}' (x,0) = 0$ therefore $\nabla_\alpha \nabla_1 g_{11} (x,0) =0$. It also means that $\nabla_1 \nabla_\alpha E_1 =0$ at $(x,0)$. For the derivative of \eqref{eq:g11-second}, changes in the order of differentiation involve curvature terms, which are well controlled. From this inspection and the subsequent Taylor expansion of $\nabla_1 g_{11}$ around $(x,0)$, we have $|\nabla_1 g_{11}(x,y)| \leq C |y|^2$ where $C$ depends on the bounds on $R$, $\nabla R$ and the dimension $n$. The estimates on $\nabla_1 g_{1\alpha}$ and $\nabla_1 g_{\alpha\beta}$ are done similarly. 

For the estimates on the second derivatives $\nabla^2_{11} g$, we repeat the process above with another derivative with respect to $\nabla_1$. Again, the fact that the Christoffel symbols are zero along $\gamma$ allows us to take ordinary derivatives and find that $\nabla^2_{11} g_{ij} (x,0) =0$ and $\nabla_{\alpha} \nabla^2_{11} g_{11}(x,0) = \nabla^2_{11} \nabla_{\alpha} g_{ij}(x,0) =0$.  The only terms not controlled previously involve $\nabla^2_{11} R$ therefore $|\nabla^2_{11} g_{ij}(x,y)| \leq C |y|^2$, with $C$ dependent of $R$, $\nabla R$, and $\nabla^2 R$. 
\end{proof}

\bibliography{references}
\bibliographystyle{alpha}

\end{document}